\documentclass[a4paper,12pt]{amsart}

\usepackage{amscd}
\usepackage{amsmath}

\usepackage{amsthm}

\usepackage{txfonts}
\usepackage{dsfont}
\usepackage{mathabx}
\usepackage{eucal}
\usepackage{tikz-cd}
\usepackage[titletoc]{appendix}

\usepackage{geometry}
\usepackage{tikz}

\newtheorem{Defn}{Definition}[section]
\newtheorem{Lemma}[Defn]{Lemma}
\newtheorem{prop}[Defn]{Proposition}
\newtheorem{theorem}[Defn]{Theorem}

\newtheorem{remark}[Defn]{Remark}
\newtheorem{Corollary}[Defn]{Corollary}

\title{A generalization of Dijkgraaf-Witten Theory}
\author{MINKYU KIM}
\date{}

\address{Graduate School of Mathematical Sciences \\ University of Tokyo}
\email{kim@ms.u-tokyo.ac.jp}

\begin{document}

\maketitle

\begin{abstract}
The main purpose of this paper is to give a generalization of Dijkgraaf-Witten theory.
We construct a TQFT for $E$-oriented manifolds using a pairing of spectra $\mu : E \wedge F \to G$ and a representative of an $F$-cohomology class of the classifying space of a finite group.
If $E= H \mathbb{Z}$, $F=G=HU(1)$ and the pairing is induced by the $\mathbb{Z}$-module structure of $U(1)$, then the TQFT reproduces Dijkgraaf-Witten theory.
For the case that each of spectra $E,F,G$ is given as the $K$-theory spectrum $KU$, we further generalize our construction based on non-commutative settings.
\end{abstract}

\tableofcontents
\section{Introduction}
The main goal of this paper is to give a generalization of Dijkgraaf-Witten theory.
Our strategy is to use a categorical group version of generalized (co)homology theory or $KK$-theory in order to construct a Lagrangian classical field theory which yields a TQFT via an integral.

Dijkgraaf-Witten invariant for $n$-manifolds is an invariant for oriented closed $n$-manifolds which is constructed from a singular cohomology class $\alpha \in H^{n} (B\Gamma ; U(1)) $ of the classifying space $B\Gamma$ of a finite group $\Gamma$ \cite{DW}, \cite{Wakui}, \cite{FQ}, \cite{Freed_CS}, \cite{FHLT}.
The class $\alpha$ induces a $U(1)$-valued characteristic number for classifying maps on closed $n$-manifolds.
Dijkgraaf-Witten invariant of an $n$-manifold $X$ is defined via an integral of the characteristic number over classifying maps on $X$.

Dijkgraaf-Witten theory is a TQFT which extends the Dijkgraaf-Witten invariant :
An (($n$-1)+1)-dimensional TQFT consists of two assignments.
It assigns a linear space to a closed ($n$-1)-manifolds and a linear map to an $n$-bordism.
We require that these assignments are compatible with each other in the sense that TQFT is given as a strong symmetric monoidal functor from $n$-cobordism category to a category of linear spaces \cite{Atiyah1}.
A functor $Z$ from a cobordism category to a linear category over a commutative field $\mathbb{F}$ yields an invariant for top dimensional closed manifolds.
The $n$-bordism $\emptyset \to X^n \leftarrow \emptyset$ in the domain cobordism category determined by closed $n$-manifold corresponds to a morphism $Z(X^n) : Z(\emptyset ) \to Z(\emptyset)$.
Moreover if we have an isomorphism $Hom ( Z (\emptyset), Z(\emptyset)) \cong \mathbb{F}$, then  $Z(X^n)$ produces a number in the ground field $\mathbb{F}$.
We say that the invariant extends to a  functor $Z$.

R. Dijkgraaf and E. Witten proposed a method to construct a (2+1)-dimensional Dijkgraaf-Witten theory starting from a singular cocycle of the classifying space $B\Gamma$ of a finite group $\Gamma$ \cite{DW}.
M. Wakui constructed (2+1)-dimensional TQFT in a rigorous way based on Dijkgraaf-Witten's idea \cite{Wakui}.
D. Freed and F. Quinn generalized Dijkgraaf-Witten theory to a higher dimensional TQFT starting from a singular cocycle of $B\Gamma$ \cite{FQ}.
They also simplified the construction of the topological action by using a canonical integral of an $n$-cocycle on an closed oriented ($n$-1)-manifold, which is valued at some torsors.
They refer to the topological action as a Lagrangian classical field theory.
A. Sharma and A. A. Voronov reformulated Dijkgraaf-Witten theory using a categorical framework \cite{SV}.
The canonical integral introduced by Freed and Quinn is well-understood under the categorical framework of Sharma and Voronov.
G. Heuts and J. Lurie generalized Dijkgraaf-Witten theory focusing on the ambidexterity of the target category of local systems on spaces of connections \cite{HL}.

We introduce a generalized Dijkgraaf-Witten invariant using a generalized cohomology class of the classifying space $B\Gamma$.
Some of main ingredients for the construction are a ring spectrum $E$, two spectra $F,G$, a pairing $E\wedge F \to G$ and an $F$-cohomology class of $\alpha \in F^n (B\Gamma)$.
We define an invariant $Z_{\varphi,\alpha}$ for closed $E$-oriented $n$-manifolds where we explain $\varphi$ in subsection \ref{subsec_DW_invariant_GCH}.
It gives a generalization of Dijkgraaf-Witten invariant in the sense that our result is reduced to Dijkgraaf-Witten invariant if we consider $E=H\mathbb{Z}$, $F=HU(1)$, the Eilenberg-Maclane spectra associated with abelian groups $\mathbb{Z}$, $U(1)$.
In that case, the pairing is induced by the $H\mathbb{Z}$-module structure of $HU(1)$.

For a ring spectrum $E$, we define a cobordism category $\mathbf{Cob}^E_{n-1}$.
We construct a TQFT $\hat{Z}_{\hat{\varphi},\hat{\alpha}}$ which is defined on the category $\mathbf{Cob}^E_{n-1}$ and extends the invariant $Z_{\varphi,\alpha}$.
Here, $\hat{\alpha}$ is a representative $\alpha$ of the class $\alpha \in F^n(B\Gamma)$.

On the one hand, we also introduce another invariant $Z_{\psi,\beta}$ for closed $KU$-oriented $n$-manifolds.
Here, $\beta$ is a $KK$-theory class in $ KK( A , C(B\Gamma) \otimes B)$ (see subsection \ref{subsec_DW_invariant_KK}) where $A,B$ are C*-algebras.
One of its advantage is that it becomes possible to cooperate with non-commutative settings.
In addition, it gives a generalization of the previous invariant $Z_{\varphi,\alpha}$ under the condition $E=F=G=KU$.
In fact, $Z_{\psi,\beta}$ coincides with $Z_{\varphi,\alpha}$ if $A= C_0 (\mathbb{R}^n)$, $B=\mathbb{C}$ and the class $\alpha$ corresponds to the class $\beta$ under the homomorphism $KU^n (B\Gamma) \to KK(C_0 (\mathbb{R}^n) , C(B\Gamma) )$.

We construct a TQFT $\hat{Z}_{\hat{\psi},\hat{\beta}}$ which is defined on the category $\mathbf{Cob}^{KU}_{n-1}$ and extends the invariant $Z_{\psi,\beta}$.
Here, $\hat{\beta}$ is a representative of the class $\beta \in KK(A, C(B\Gamma) \otimes B)$.

Since the invariants $Z_{\varphi,\alpha}$ and $Z_{\psi,\beta}$ coincide to each other under some conditions mentioned before, the dimension of linear spaces assigning to ($n$-1) closed manifolds coincide to each other.
In fact, $Z (Y^{n-1} \times S^1)$ coincides with the dimension of the linear space $Z (Y^{n-1})$ if $Z = Z_{\varphi,\alpha} , Z_{\psi,\beta}$ is defined on a closed ($n$-1)-manifold $Y^{n-1}$.
Nonetheless, we do not have a natural isomorphism between $\hat{Z}_{\hat{\varphi},\hat{\alpha}}$ and $\hat{Z}_{\hat{\psi},\hat{\beta}}$.
It seems necessary to construct a `map' between classes of representatives of $KU^n(B\Gamma)$ and $KK(S^n , C(B\Gamma))$ in an appropriate way, but we do not yet have such a `map'.

Note that the TQFT we construct is {\it partial} unless the generalized (co)homology theory satisfies some condition which ordinary (co)homology theory satisfies automatically.
The object class of its domain cobordism category $\mathbf{Cob}^E_{n-1}$ consists of ($n$-1)-manifolds $Y^{n-1}$ such that $E_n (Y^{n-1}) \cong 0$ where $n$ is the top dimension of manifolds we deal with.

The construction basically follows the approach of Sharma-Voronov.
Sharma-Voronov use categorical (co)homology group to extend characteristic number of $n$-dimensional closed $B\Gamma$-manifolds to $(n-1)$ dimension where the extended characteristic number is called Lagrangian classical field theory in the literature.
We formulate a categorical framework for generalized (co)homlogy theory and $KK$-theory.
The categorical group version of singular (co)homology theory applied by A. Sharma and A. A. Voronov is due to  A. del R\`{i}oa, J. Mart\`{i}nez-Morenob, E. M. Vitale \cite{RMV}.
We introduce a categorical group version of generalzied (co)homology theory and $KK$-theory based on a fundamental groupoid of certain simplicial sets.
Sharma and Voronov constructed a Lagrangian classical field theory using the cap product associated with that categorical (co)homology theory.
We use a pairing of generalized (co)homology theories or the Kasparov product of $KK$-theory to construct a Lagrangian classical field theory.

We only explain the construction of TQFT starting from a generalized cohomology class.
We give a categorical framework of KK-theory which is used to construct the TQFT using a KK-theory class in a parallel way given in this paper.

In this paper, we use the notion of spectrum and smash product following Adams \cite{Adams}, Switzer \cite{Switzer}.
$E= \{ E_n , s_n \}$ is a spectrum if $E_n$ is a pointed CW-complex and $s_n : \Sigma E_n \to E_{n+1}$ is a pointed CW-embedding for $n \in \mathbb{Z}$.
There are some natural equivalences under which the smash product is associative, commutative, and has the sphere spectrum as a unit upto coherent natural equivalence.

A. del R\`{i}oa, J. Mart\`{i}nez-Morenob, E. M. Vitale \cite{RMV} call by (co)homology categorical group the categorical group version of (co)homology theory induced by a chain complex of categorical groups.
We follow the convention by calling by generalized (co)homology categorical group (resp. $KK$ categorical group) the categorical group version of generalized (co)homology group ($KK$-group).

The organization of this paper is as follows.
In section \ref{Notations}, we introduce some notations which appear very often throughout this paper.
In section \ref{section_Generalized Dijkgraaf-Witten invariants}, we define generalized Dijkgraaf-Witten invariants.
In section \ref{section_main_results}, we give the main results of this paper without precise meaning of notations.
From section \ref{Generalized (co)homology theory SCG} to section \ref{Quantum theory from generalized (co)homology}, we construct a TQFT based on generalized (co)homology theory.
In section \ref{Generalized (co)homology theory SCG}, we introduce a symmetric categorical group version of generalized (co)homology theory.
In section \ref{Classical theory from generalized (co)homology}, we construct a classical field theory using generalized (co)homology theory.
In section \ref{Quantum theory from generalized (co)homology}, we apply push-forward to the result in section \ref{Classical theory from generalized (co)homology} in order to obtain a TQFT for $E$-oriented manifolds.
In section \ref{compute_examples}, we give some computations of the TQFT's.
In section \ref{commets_about_rho}, we give a remark about the symmetric monoidal functors $\hat{\varphi},\hat{\psi}$ which are used to construct the TQFT's.
In section \ref{section_categorical_KK_theory}, we introduce a symmetric categorical group version of $KK$-theory.

\section*{Acknowledgements}
The author is deeply grateful to his supervisor Mikio Furuta who gave him great opportunities to meet various mathematical objects.
In particular, the author's interest in $K$-theory and C*-algebras is due to Mikio Furuta.
The author expresses gratitude to Yosuke Kubota who devoted his precious time to give the author valuable advices about how to write a research paper.
The author was also helped by Yosuke Kubota to deal with some technical problems when he was trying to  introduce a categorical group version of $KK$-theory based on the Kasparov picture in the first draft.

\section{Notations}
\label{Notations}
\begin{itemize}
\item
Let $X,Y$ be spaces.
We denote by $Map (X,Y)$ the set of maps from $X$ to $Y$.
We denote by $MAP(X,Y)$ the simplicial set whose $k$-skeleton $MAP(X,Y)_k$ is given by 
\begin{align}
Map ( X \times \triangle (k) , Y) ,
\end{align}
where $\triangle (k)$ is the $k$-simplex.
Then the face maps and degeneracy maps on simplices induce face maps and degeneracy maps of $MAP(X,Y)_k$'s.
Similarly, we define a simplicial set $MAP_\ast (X,Y)$ for based spaces $X,Y$ where the $k$-skeleton is given by 
\begin{align}
Map_\ast (X\wedge (\triangle (k))^+ , Y) .
\end{align}

\item
 About bicategories, we follow the terminologies introduced in \cite{T.Leinster}.
 We introduce following notations of several categories.
  \begin{enumerate}
  \item
  $\mathbf{Gpd}$ is the 2-category of groupoids, functors and natural transformations.  
  \item
  $\mathbf{CW}_\ast$ is the 2-category whose objects are given as based CW-spaces and the morphism category $\mathbf{CW}_\ast (X,Y)$ is given as the fundamental groupoid $\Pi_1 MAP_\ast (X,Y)$ where $MAP_\ast (X,Y)$ is the simplicial set consisting of based maps from $X$ to $Y$.
  \item
  $\mathbf{2Cat}$ is a 2-category of 2-categories, strict homomorphisms and strict transformations.
  \item
  $\mathbf{Vect}_{\mathbb{F}}$ is the category of finite-dimensional vector spaces over the field $\mathbb{F}$ and linear homomorphisms.
  \item 
  For an abelian group $M$, we define a symmetric categorical group $M\mathbf{Tor}$ as follows.
  Its object class consists of $M$-torsors (i.e. principal $M$-bundle over a point).
  Its morphism class consists of $M$-equivariant maps between $M$-torsors.
  For two $M$-torsors $T,T^\prime$, we define $T \otimes T^\prime$ obtained by dividing the product $T \times T^\prime$ by $(t,t^\prime) \sim (mt , mt^\prime),~m \in M$.
  Then the assignment $(T,T^\prime) \mapsto T\otimes T^\prime$ determines a functor $\otimes : M\mathbf{Tor} \times M\mathbf{Tor} \to M \mathbf{Tor}$ which gives a symmetric categorical group structure to $M\mathbf{Tor}$ in an obvious way.
  \end{enumerate}
\end{itemize}


\section{Generalized Dijkgraaf-Witten invariants}
\label{section_Generalized Dijkgraaf-Witten invariants}

Let $\mathbb{F}$ be a commutative field with characteristic zero.
In this section, we define two $\mathbb{F}$-valued invariants which are denoted by $Z_{\varphi,\alpha}$ and $Z_{\psi,\beta}$ throughout this paper.
The invariant $Z_{\varphi,\alpha}(X^n)\in \mathbb{F}$ is defined for closed $E$-oriented $n$-manifolds $X^n$ where $E$ is a ring spectrum.
Its definition also depends on choices of a pairing of spectr $\mu : E \wedge F \to G$, an $F$-cohomology class $\alpha \in F^n (B\Gamma)$ and a group homomorphism $\varphi : G_0 (\mathrm{pt}) \to  \mathbb{F}^{\times}$.
If we consider $E= H\mathbb{Z}$, $F=G=HU(1)$, $\mathbb{F}= \mathbb{C}$ and the inclusion $\varphi : G_0 (\mathrm{pt}) \cong U(1) \to \mathbb{C}^{\times}$, then the invariant $Z_{\varphi,\alpha}$ reproduces the Dijkgraaf-Witten invariant \cite{DW}, \cite{FQ}.
Due to this fact, we regard it as a generalization of the Dijkgraaf-Witten invariant.
The invariant $Z_{\psi,\beta}(X^n)\in \mathbb{F}$ is defined for closed $KU$-oriented $n$-manifolds.
Its definition also depends on choices of two C*-algebras $A,B$, a $KK$-group class $\beta \in KK(A, C(B\Gamma)\otimes B)$ and a group homomorphism $\psi : KK(A, S^n B) \to \mathbb{F}^{\times}$.
Here, $KK(A, S^n B)$ is the $KK$-group corresponding to C*-algebras $A$ and $S^n B$.

\subsection{An invariant $Z_{\varphi,\alpha}$ from generalized (co)homology theory}
\label{subsec_DW_invariant_GCH}

In this paper, we use the notion of spectrum following Adams \cite{Adams}, Switzer \cite{Switzer}.
The notion of spectra is actually not that essential in this section since we only need the generalized (co)homology theories induced by spectra.
Nonetheless, we use it in order to relate the invariants defined in this section with TQFT's constructed later.

Let $E$ be a ring spectrum.
Suppose that we are given a pairing of spectra $\mu : E\wedge F \to G$..
Then the pairing $\mu$ induces a Kronecker pairing $\langle ~,~ \rangle : E_n (B\Gamma) \otimes F^n (B\Gamma) \to G_0 (\mathrm{pt})$.

Let $\alpha \in F^n (B\Gamma)$ be an $F$-cohomology class of the classifying space.

\begin{Defn}
\label{CFT_GCH_Salpha_defn}
For an $E$-oriented closed n-manifold $X^n$ and a continuous map $f : X^n \to B\Gamma$, we define 
\begin{align}
S_{\alpha} (X^n , f) \stackrel{\mathrm{def.}}{=} \langle f_\ast [X^n]_E , \alpha \rangle \in G_0 (\mathrm{pt}) ,
\end{align}
where $[X^n]_E \in E_n (X^n)$ is the fundamental class of the $E$-oriented manifold $X^n$.
\end{Defn}

\begin{Defn}
We define two abelian monoids $\mathbf{M}^{E}_n (T)$ and $\mathbf{M}^n_{E} (T ; N)$ where $T$ is a space and $N$ is an abelian monoid :
\begin{align}
\mathbf{M}^{E}_n (T) \stackrel{\mathrm{def.}}{=} \{ (X^n , f : X^n \to T) \} / \sim , \notag \\
\mathbf{M}^n_{E} (T ; N) \stackrel{\mathrm{def.}}{=} Hom (\mathbf{M}^{E}_n (T)  , N) . \notag
\end{align}
The equivalence relation $\sim$ is defined by 
\begin{align}
(X, f) \sim (X^\prime, f^\prime) \Leftrightarrow \exists \alpha : X \stackrel{\cong}{\to} X^\prime ~~ \mathrm{s.t.}~~ f^\prime \circ \alpha \simeq f . \notag
\end{align}
Following Turaev \cite{Turaev_HQFT}, we call the pair $(X^n , f : X^n \to T)$ as a $T$-manifold and $\mathbf{M}^{E}_n (T) $ the set of equivalence classes of $T$-manifolds.
We give the set $\mathbf{M}^{E}_n (T)$ an abelian monoid structure using the disjoint union.
Then $\mathbf{M}^n_{E} (T ; N)$ is the set of homomorphisms between monoids $\mathbf{M}^{E}_n (T)  , N$.
\end{Defn}

It is easy to check that the assignment $T \mapsto \mathbf{M}^{E}_n (T)$ is covariant so that $T \mapsto \mathbf{M}^n_{E} (T; N)$ is contravariant.
For a continuous map $f : T \to T^\prime$, we denote the induced maps by $f_\ast : \mathbf{M}^{E}_n (T) \to \mathbf{M}^{E}_n (T^\prime)$, $f^\ast : \mathbf{M}^n_{E} (T^\prime ; N) \to \mathbf{M}^n_{E} (T ; N) $.

\begin{Defn}
\label{pushforward_!_def}
Let $c: B\Gamma \to \mathrm{pt}$ be the collapsing map.
Consider the field $\mathbb{F}$ as an abelian monoid using its multiplication structure.
We define a map $c_{!} : \mathbf{M}^n_{E} (B\Gamma ; \mathbb{F}) \to \mathbf{M}^n_{E} (\mathrm{pt} ; \mathbb{F})$ by
\begin{align}
c_{!} (I) \stackrel{\mathrm{def.}}{=} \sum_{f: X^n \to B\Gamma} \frac{I (X^n , f)}{\sharp Aut (f^\ast (E\Gamma))} \in \mathbb{F} \notag.
\end{align}
Here, $f$ runs on the set of homotopy classes of continuous maps.
\end{Defn}

In order to check that this definition is well-defined, one should check that $c_{!}(I)$ gives a homomorphism from $\mathbf{M}^E_n (\mathrm{pt})$ to the multiplicative monoid $\mathbb{F}$.
Note that $I (X^n_0 , f_0) \cdot I(X^n_1 , f_1) = I (X^n_0 \amalg X^n_1 , f_0 \amalg f_1) \in F_0 (\ast)$ since $I \in \mathbf{M}^n_E (B\Gamma ; \mathbb{F})$ so that we obtain
\begin{align}
c_{!}(I) (X^n_0) \cdot c_{!}(I) (X^n_1) = c_{!}(I) (X^n_0 \amalg X^n_1 ) \in \mathbb{F}.
\end{align}

We define a generalized Dijkgraaf-Witten invariant for closed $E$-oriented manifolds as follows.

\begin{Defn}
\label{TQFT_GCH_Zalpha_defn}
Let $\varphi : F_0 (\mathrm{pt} ) \to \mathbb{F}^{\times}$ be a group homomorphism.
We define $S_{\varphi,\alpha} \stackrel{\mathrm{def.}}{=} \varphi_\ast (S_{\alpha} )$ and $Z_{\varphi,\alpha} \stackrel{\mathrm{def.}}{=} c_! (S_{\varphi,\alpha})$.
In other words, for an $E$-oriented closed n-manifold $X^n$, we define
\begin{align}
Z_{\varphi,\alpha} (X^n) \stackrel{\mathrm{def.}}{=} \sum_{f: X^n \to B\Gamma} \frac{\varphi(S_{\alpha} (X^n , f))}{\sharp Aut (f^\ast (E\Gamma))} \in \mathbb{F} .
\end{align}
\end{Defn}

\subsection{An invariant $Z_{\psi,\beta}$ from $KK$-theory}
\label{subsec_DW_invariant_KK}

$K$-theory for C*-algebras and analytic $K$-homology are generalized into a common framework called $KK$-theory \cite{Blackadar}.
$KK$-theory is a bivariant functor $(A,B) \mapsto KK(A,B) $ from the category of C*-algebras and homomorhisms to the category of abelian groups and homomorphisms, which is contravariant with respect to $A$ and covariant with respect to $B$.

We have a natural isomorphism $KU_k (Z,Z^\prime) \cong KK(C_0 (Z \backslash Z^\prime),C_0 (\mathbb{R}^k))$ for a pair of finite CW complexes $(Z,Z^\prime)$ (Theorem 4 \cite{Kasparov}).
Here the left hand side is the homology theory defined by the spectrum $KU$.
If $Z^k$ is a compact $KU$-oriented $k$-manifold, then its $KU$-orientation lives in the $KK$-group $KK(C(Z^k \backslash \partial Z^k) , C_0 (\mathbb{R}^k))$.
We denote it by $[Z^k]_{aKU}$.

As a technical assumption, we fix a CW complex structure of the classifying space $B\Gamma$, whose $r$-skeleton $B\Gamma^{(r)}$ is compact and Hausdorff for every $r\in \mathbb{Z}^{>0}$.
Since a smooth $k$-manifold $Z^k$ has a $k$-dimensional CW complex structure, the celluar approximation theorem induces an equivalence of groupoids,
\begin{align}
\Pi_1 MAP (Z^k , B\Gamma^{(r)}) \to \Pi_1 MAP (Z^k , B\Gamma),~r \geq k+2.
\end{align}
We fix a large $r \geq n+3$ where $n+1$ is the largest dimension of manifolds which we deal with in this paper.
Let us denote by $\Pi_1 MAP^f (Z^k , B\Gamma) \stackrel{\mathrm{def.}}{=} \Pi_1 MAP (Z^k , B\Gamma^{(r)})$.
For such $r$, we denote
\begin{align}
KK(A,C(B\Gamma ) \otimes B) \stackrel{\mathrm{def.}}{=} KK(A, C(B\Gamma^{(r)} ) \otimes B) .
\end{align}
This assumption is necessary since the *-algebra formed by $\mathbb{C}$-valued continuous maps on $B\Gamma$ does not induce a C*-algebra in general.

Let $A,B$ be C*-algebras and let $\beta \in KK(A ,C(B\Gamma) \otimes B)$.
We define an action funcional for principal $\Gamma$-bundles over closed $KU$-oriented $n$-manifolds.

\begin{Defn}
\label{TQFT_KK_Sbeta_defn}
Let $X^n$ be a closed $KU$-oriented n-manifold.
For an object $f \in \Pi_1 MAP^f (X^n, B\Gamma)$, we define 
\begin{align}
S_{\beta} (X^n , f) \stackrel{\mathrm{def.}}{=} f_\ast [X^n]_{aKU} \otimes_{C(X^n)} \beta  \in KK(A , S^n B) 
\end{align}
where $\otimes_{C(X^n)}$ is the Kasparov product contracting $C(X^n)$.
\end{Defn}

We define a generalized Dijkgraaf-Witten invariant $Z_{\psi,\beta}$ for closed $KU$-oriented $n$-manifolds as follows.
Again, one can check that $Z_{\psi,\beta}  \in \mathbf{M}^n_{KU} (\mathrm{pt} ; \mathbb{F})$.

\begin{Defn}
\label{TQFT_KK_Zbeta_defn}
Let $\psi : KK(A, S^n B) \to \mathbb{F}^{\times}$ be a group homomorphism.
We define $S_{\psi,\beta} \stackrel{\mathrm{def.}}{=} \psi_\ast (S_{\beta})$ and $Z_{\psi,\beta} \stackrel{\mathrm{def.}}{=} c_! (S_{\psi,\beta})$.
In other words, for a closed $KU$ oriented $n$-manifold $X$, we define
\begin{align}
Z_{\psi,\beta} (X^n) \stackrel{\mathrm{def.}}{=} \sum_{[f]} \frac{\psi (S_{\beta} (X^n , f))}{\sharp Aut (f^\ast (E\Gamma))} \in \mathbb{F} .
\end{align}
Here, $f$ runs on the set $\pi_0 ( MAP^f (X^n , B\Gamma)) \cong [X^n , B\Gamma^{(r)}]$.
\end{Defn}


\section{Main results}
\label{section_main_results}

In this section, we outline our main results without giving precise definition of notations.
In subsection \ref{mainresult A TQFT from generalized (co)homology theory}, we give one of our main theorems that the invariant $Z_{\varphi,\alpha}$ extends to a TQFT $\hat{Z}_{\hat{\varphi},\hat{\alpha}}$.
We give explanation that it yields the Dijkgraaf-Witten theory \cite{FQ} as its corollary.
In subsection \ref{mainresult A TQFT from KK-theory}, we give our other main theorem that the invariant $Z_{\psi,\beta}$ extends to a TQFT $\hat{Z}_{\hat{\psi},\hat{\beta}}$.
In subsection \ref{subsection_summary_of_const}, we summarize their constructions.
In subsection \ref{main_feat_construction}, we explain the main feature of our construction.

\subsection{A TQFT $\hat{Z}_{\hat{\varphi},\hat{\alpha}}$ from generalized (co)homology theory}
\label{mainresult A TQFT from generalized (co)homology theory}

Let $E$ be a ring spectrum.
We denote by $\mathbf{Cob}^E_{n-1}$ a cobordism category of $E$-oriented smooth manifolds whose top dimension is $n \in \mathbb{Z}^{>0}$.
Its object class consists of ($n$-1)-dimensional closed $E$-oriented smooth manifolds $Y^{n-1}$ such that
\begin{align}
E_n (Y^{n-1} ) \cong 0 .
\end{align}
Here, $E_n$ denotes the $n$-th $E$-homology group.
The morphism class of the category $\mathbf{Cob}^E_{n-1}$ consists of $n$-dimensional smooth $E$-oriented cobordisms connecting such objects.
Due to the additivity of $E$-homology with respect to disjoint union of spaces, the disjoint union of $E$-oriented manifolds induces a symmetric monoidal structure on the category $\mathbf{Cob}^E_{n-1}$.

In section \ref{Generalized (co)homology theory SCG}, we construct a delooping of the abelian group $G_0 (\mathrm{pt})$ which is given by a symmetric categorical group $\mathcal{G} (= \hat{G}_{-1}(S^0))$ such that 
\begin{align}
\pi_1 (\mathcal{G}) \cong G_0 (\mathrm{pt}) ,
\end{align}
as abelian groups.
Then our main theorem is stated as follows.

\vspace*{3mm}
{\bf Main theorem 1.}(see Theorem \ref{corollary_from_cob_tqft_GCH})
Let $\alpha$ be a representative of the $F$-cohomology class $\alpha \in F^n (B\Gamma)$.
Let $\hat{\varphi} : \mathcal{G} \to  \mathbb{F}^{\times} \mathbf{Tor}$ be a symmetric monoidal functor such that $\pi_1 (\hat{\varphi})= \varphi :  G_0 ( \mathrm{pt} ) \to \mathbb{F}^{\times}$.
The invariant $Z_{\varphi,\alpha}$ for closed $E$-oriented $n$-manifolds extends to a strong symmetric monoidal functor,
\begin{align}
\hat{Z}_{\hat{\varphi},\hat{\alpha}} : ( \mathbf{Cob}^E_{n-1} , \amalg) \to ( \mathbf{Vect}_{\mathbb{F}} , \otimes ) .
\end{align}
\vspace*{3mm}

It is a generalization of the Dijkgraaf-Witten theory \cite{DW},\cite{FQ} :
We substitute the following spectra into $E,F$,
\begin{align} 
E = H \mathbb{Z} , ~ F =G = HU(1) .
\end{align}
Here $HM$ is the Eilenberg Maclane spectrum associated with an abelian group $M$.
Let $\varphi :U(1) \cong  G_0 (\mathrm{pt}) = H_0 (\mathrm{pt}  ; U(1)) \to M = \mathbb{C}^{\times}$ be the inclusion.
Then it is lifted to an equivalence of symmetric categorical groups, $\hat{\varphi} : \mathcal{G} \to \mathbb{C}^{\times}\mathbf{Tor}$ (see section \ref{commets_about_rho}).
Under these assumptions, the strong symmetric monoidal functor $\hat{Z}_{\hat{\varphi}, \hat{\alpha}}$ in the theorem gives the Dijkgraaf-Witten theory.
In fact, the cobordism category $\mathbf{Cob}^E_{n-1}$ is isomorphic to the cobordism category of oriented manifolds since any closed $k$-manifold $Z^k$ satisfies $H_{k+1}(Z^k ; \mathbb{Z}) \cong 0$.

\subsection{A TQFT $\hat{Z}_{\hat{\psi},\hat{\beta}}$ from $KK$-theory}
\label{mainresult A TQFT from KK-theory}

Let $A,B$ be C*-algebras.
In section \ref{section_categorical_KK_theory}, we construct a delooping of the abelian group $KK(A, S^n B)$ which is given as a symmetric categorical group $\mathcal{G} (= \mathcal{KK}(A, S^{n+1}B))$ such that 
\begin{align}
\pi_1 (\mathcal{G}) \cong KK(A, S^n B),
\end{align}
as abelian groups.

\vspace*{3mm}
{\bf Main Theorem 2.}
Let $A,B$ be C*-algebras.
Let $\beta$ be a representative of the $KK$-theory class $\beta \in KK(A, C(B\Gamma) \otimes B)$.
Let $\hat{\psi} : \mathcal{G} \to \mathbb{F}^{\times} \mathbf{Tor}$ be a symmetric monoidal functor such that $\pi_1 (\hat{\psi})= \psi :  KK(A, S^n B) \to \mathbb{F}^{\times}$.
The invariant $Z_{\psi,\beta}$ extends to a strong symmetric monoidal functor,
\begin{align}
\hat{Z}_{\hat{\psi},\hat{\beta}} : ( \mathbf{Cob}^{KU}_{n-1} , \amalg) \to ( \mathbf{Vect}_{\mathbb{F}} , \otimes ) .
\end{align}

\subsection{Outline of the construction}
\label{subsection_summary_of_const}

We summarize the construction of TQFT's.
In order to extend our Dijkgraaf invariant to a TQFT, we firstly extend our Lagrangian classical field theory or characteristic number of $B\Gamma$-manifolds.
As the Dijkgraaf-Witten invariant is obtained using the push-forward in Definition \ref{pushforward_!_def}, we apply a push-forward $c_{!} : \mathbf{Cob}^{n-1}_E (B\Gamma ; \mathbf{Vect}_{\mathbb{F}}) \to \mathbf{Cob}^{n-1}_E (\mathrm{pt} ; \mathbf{Vect}_{\mathbb{F}})$ to the extended characteristic number to obtain a TQFT.

We only consider the generalized cohomology case here, but it is possible to discuss the KK-theory case in a parallel way.
It is helpful to recall the definition of the invariant $Z_{\varphi,\alpha}$.
For each closed $n$-manifold $X^n$ which has an appropriate orientation with respect to some spectrum, we are given a set of finite gauge fields over which a Lagrangian classical field theory $S_{\varphi,\alpha}$ is defined.
$S_{\varphi,\alpha}$ is understood as a characteristic number of closed $B\Gamma$-manifolds obtained by composing following maps :
\begin{align}
\mathbf{M}^E_n (B\Gamma ) \to E_n (B\Gamma) \stackrel{\langle \alpha , - \rangle}{\to} G_0 (\mathrm{pt}) \stackrel{\varphi}{\to} \mathbb{F}^{\times}
\end{align}
The first step is that we extend the characteristic number $S_{\varphi,\alpha}$ to closed $(n-1)$-manifolds.
We denote it by $\hat{S}_{\hat{\varphi},\hat{\alpha}}$, which is defined by composing following functors,
\begin{align}
\mathbf{Cob}^{E,rep}_{n-1} (B\Gamma ) \stackrel{\Phi_{B\Gamma}}{\to} \hat{E}_{n+1} (\Sigma^2 B\Gamma^+) \stackrel{\langle \hat{\alpha} ,- \rangle}{\to} \hat{G}_{-1}(S^0) \stackrel{\hat{\varphi}}{\to} \mathbb{F}^{\times}\mathbf{Tor} .
\end{align}
Here, the symmetric monoidal functor $\Phi_{B\Gamma}$ is defined in section \ref{Classical theory from generalized (co)homology}.

Dijkgraaf-Witten invariant $Z_{\varphi,\alpha}(X^n)$ is defined via an integral of the functional $S_{\varphi,\alpha}$ over finite gauge fields.
Our strategy to construct TQFT's is parallel to the construction of the invariant $Z_{\varphi,\alpha}$.
In fact, there exists a push-forward $c_{!} : \mathbf{Cob}^n_E (B\Gamma ; \mathbf{Vect}_{\mathbb{F}}) \to \mathbf{Cob}^n_E (\mathrm{pt} ; \mathbf{Vect}_{\mathbb{F}})$, which is called finite path integral in the literature.
The symmetric monoidal functor $c_{!} \hat{S}_{\hat{\varphi},\hat{\alpha}} \in \mathbf{Cob}^n_E (\mathrm{pt} ; \mathbf{Vect}_{\mathbb{F}})$ obtained by applying it to the extended characteristic number is the TQFT we want.

We should give a remark about $\hat{E}_\bullet$.
In order to extend invariants to codimension-one manifolds, we start with a categorical framework of abelian groups, i.e. symmetric categorical groups.

We lift the generalized (co)homology groups and $KK$-theory groups to categorical groups as follows.
Both of generalized (co)homology groups and $KK$-theory groups can be defined as a homotopy set $\pi_0 (Mor (a,b))$ of a space $Mor (a,b)$ for appropriate object $a,b$, which are spectra for the former case and C*-algebras for the latter case.
Our categorical group version of generalized (co)homology groups and $KK$-theory groups are constructed on the fundamental groupoid $\Pi_1 Mor (a,b)$.

\subsection{Main feature of the construction}
\label{main_feat_construction}

In previous research on Dijkgraaf-Witten theory \cite{DW} \cite{Wakui} \cite{FQ} \cite{SV}, the functor $\Phi_{B\Gamma}$ is constructed implicitly by using simplicial data of manifolds.
Triangulations of manifolds and the classifying space are used in Dijkgraaf-Witten, Wakui, and the singular nerve of manifolds is used in Freed-Quinn, Sharma-Voronov.
In particular, the morphism assigned by $\Phi_{B\Gamma}$ is induced by the natural transformation associated in the 2-exact sequence of categorical (singular) homology groups in Sharma-Voronov.
The main feature of our construction is that we do not need such auxiliary data. 
We construct corresponding morphisms by using a homotopy associated with every bordism.
In this subsection, let us explain the idea.

For a moment, we deal with arbitrary manifolds.
Let $Y$ be a closed $(n-1)$-manifold.
We denote by $B^{\prime}_0 (Y)$ the quotient space $Y \times D^1 / Y \times \partial D^1$ where $D^1 \stackrel{\mathrm{def.}}{=} \{ t \in \mathbb{R} ~|~ -1 \leq t \leq 1 \}$.
For a $n$-bordism $X$ from $Y_0$ to $Y_1$, we construct $B^{\prime}_1 (Y_0 \stackrel{i_0}{\to} X \stackrel{i_1}{\leftarrow} Y_1) = B^{\prime}_1 (X)$ as follows.
We glue the boundaries of $Y_1 \times [-1, 0] \amalg X \times \{0 \} \amalg Y_1 \times [0,1]$ to obtain a bundle $L$ over $D^1$, i.e. we identify,
\begin{itemize}
\item
$(y_1 , 0 ) \sim (i_1 (y_1) ,0)$ where $y_1 \in Y_1$
\item
$(y_0 , 0 ) \sim (i_0 (y_0) ,0)$ where $y_0 \in Y_0$
\end{itemize}
Let us write the projection to $D^1$ by $\pi : L \to D^1$.
We define $B^{\prime}_1 (X)$ as the quotient space $L / \pi (\partial D^1)$.
Then $B^{\prime}_0 (Y)$, $B^{\prime}_1(X)$ are naturally determined by $Y,X$ respectively.
We have a homotopy $h$ denoted in the following diagram,
\begin{equation}
\label{B_diamond_XY}
\begin{tikzcd}
 & B^{\prime}_1(Y_0 \to X \leftarrow Y_1) \arrow[ld, "q_1"] \arrow[rd, "q_0"] &  \\
B^{\prime}_0 (Y_0)  \arrow[rd] & \stackrel{h}{\Longrightarrow} & B^{\prime}_0 (Y_1) \arrow[ld] \\
 & B^{\prime}_0(X) & 
\end{tikzcd}
\end{equation}
The upper two maps are obtained by collapsing $Y_1 \times [-1,0]$ and $Y_0 \times [0,1]$ respectively.
The lower two maps are induced from the inclusions.

A continuous map $g : Y \to B\Gamma$ induces a map $B^{\prime}_0 (g) : B^{\prime}_0 (Y) \to B^{\prime}_0 (B \Gamma)$.
Let us consider a bordism between $B\Gamma$-manifolds.
In other words, we are given a bordism $X$ from $Y_0$ to $Y_1$ and continuous maps $g_k : Y_k \to B\Gamma$, $f : X \to B\Gamma$ such that the restriction of $f$ to boundaries coincides with $g_0 \amalg g_1$.
Then we can construct a homotopy from $B^{\prime}_0 (g_0) \circ q_1$ to $B^{\prime}_0 (g_1) \circ q_0$ by using the previous homotopy $h$.

Let $E = \{E_n , s_n \}$ be a spectrum.
From now on, we consider $E$-oriented manifolds and bordisms.
An $(n-1)$-th $E$-homology class of $Y$ is represented by a map from the sphere spectrum $S^{n+1}$ to $E \wedge \Sigma^{\infty} \Sigma^2 Y^+$.
Note that we have a homeomorphism between spaces $\Sigma^2 Y^+$ and $B^{\prime}_0 (Y)$ defined above.
Hence, we consider a representative of an $(n-1)$-th $E$-homology class of $Y$ as a map to $E \wedge \Sigma^{\infty} B^{\prime}_0 (Y)$.
Then a continuous map $g : Y \to B\Gamma$ and a representative of an $E$-orientation of $Y$ (as an $E$-homology class) gives a map $S(Y,g)$ from the sphere spectrum $S^{n+1}$ to $E \wedge \Sigma^{\infty} B^{\prime}_0 (B\Gamma)$.
For an $E$-oriented $B\Gamma$-bordism $(X, f)$ between $E$-oriented $B\Gamma$-manifolds $(Y_0, g_0),(Y_1,g_1)$, we obtain a homotopy $S(X,f)$ from $S(Y_0,g_0)\circ q_1$ to $S(Y_1,g_1) \circ q_0$ :
$$
\begin{tikzcd}
 & B^{\prime}_1(Y_0 \to X \leftarrow Y_1) \arrow[ld, "q_1"] \arrow[rd, "q_0"] &  \\
B^{\prime}_0 (Y_0)  \arrow[rd, "S(Y_0g_0)"] & \stackrel{S(X,f)}{\Longrightarrow} & B^{\prime}_0 (Y_1) \arrow[ld, "S(Y_1g_1)"] \\
 & B^{\prime}_0(B\Gamma) & 
\end{tikzcd}
$$
The above constructions might be sufficient to understand what are assigned to objects and morphisms from the TQFT's, however it is hard to check whether the assignments give a functor, i.e. whether commutative diagrams in the domain are sent to commutative diagrams in the target.
It is the reason that we use bundles over the 2-disk $D^2$ not over the 1-disk $D^1$ in Definition \ref{defn_2disk_B}.

Let us recall how the homotopy $h$ is obtained.
One might realize that the maps in the diagram (\ref{B_diamond_XY}) are nothing but the collapsing maps $\pi_u$ from $D^1 / \partial D^1$ to $U/\partial U \cong S^1$ where $U$ is a subspace of $D^1$ with a homeomorphism $u$ to the 1-disk $D^1$ if one focuses on their base spaces.
Then the homotopy $h$ is obtained from a homotopy between $\pi_u$ and $\pi_v$ where we consider another embedding $v : D^1 \to D^1$.
More generally, given a manifold $T^n$ and an embedding $u : D^n \to T^n$, we have such collapsing map $\pi_u : T^n \to D^n / \partial D^n \cong S^n$, and if two embeddings are connected by an isotopy $i$, then $i$ induces a homotopy between $\pi_u$ and $\pi_v$.

\section{Generalized (co)homology categorical group}
\label{Generalized (co)homology theory SCG}

In this section, we introduce a categorical group version of generalized (co)homology theory
It is defined as the fundamental groupoid of a space of morphisms between spectra whose 0-th homotopy set gives the generalized (co)homology group.
The abelian group structure on generalized (co)homology group is lifted to the groupoid as a symmetric categorical group structure.

In subsection \ref{preli_Generalized (co)homology theory SCG},  we introduce a simplicial set consisting of morphisms between spectra, and discuss some properties of its fundamental groupoid.
In subsection \ref{categorical cohomology group}, we define categorical generalized (co)homology group denoted by $\hat{E}_n (X)$ and $\hat{E}^n(X)$ for a pointed space $X$.

\subsection{Preliminaries}
\label{preli_Generalized (co)homology theory SCG}

We follow \cite{Adams}, \cite{Switzer}  with respect to the notion of spectra.
We say that $E=\{ E_n , s_n\}$ is a spectrum if $E_n$'s are pointed CW complexes and $s_n : \Sigma E_n \to E_{n+1}$'s are pointed CW-embeddings.
For two spectra $E,F$, a map is a family of celluar maps $f_n : E_n \to F_n$ between CW-complexes which intertwines the structure maps $s_n$'s.
For two spectra $E,F$, a morphism is an equivalence class of a pair $(f,E^\prime) : E \to F$ where $E^\prime \subset E$ is cofinal subspectrum and $f : E^\prime \to F$ is a map of spectra.
Here, $E^\prime \subset E$ is cofinal if there exists $k \in \mathbb{N}$ such that $\Sigma^k e \subset E^\prime_{n+k}$ for any cell $e$ of $E_n$.
Two pairs $(f,E^\prime ), (g , E^{\prime\prime} )$ are equivalent if $f |_{E^\prime \cap E^{\prime\prime} }= g|_{E^\prime \cap E^{\prime\prime}} $.
We denote by $Mor(E,F)_n$ the set of morphisms of degree $n$ from $E$ to $F$.

\begin{Defn}
We define a simplicial set $MOR (E,F)_n = X$ for two spectra $E,F$.
Let its $m$-skeleton $X_m$ be $Mor (E \wedge \triangle (m)^+ , F)_n$.
Here, $E\wedge \triangle (n)^+ = \{ E_n \wedge \triangle (n)^+ \}$ is given by the smash product of a spectrum and a pointed CW-complex.
We use the face maps and degeneracy maps between simplices $\triangle (n)$ to define the face maps and degeneracy maps between $X_m$'s.
\end{Defn}

For two subspectra $E^\prime, E^{\prime\prime}$ of $E$, if two morphisms $f^{\prime} : E^\prime \to F$, $f^{\prime\prime} : E^{\prime\prime} \to F$ coincide on $E^\prime \cap E^{\prime\prime}$, then they are glued into a unique morphism $f : E^\prime \cup E^{\prime\prime} \to F$.
Hence, we see that the simplicial set $MOR(E,F)_n$ satisfies the Kan condition \cite{Kan}, \cite{Kan2} so that its fundamental groupoid of $MOR(E,F)$ is well-defined.

Next, we lift the composition of morphisms to a simplicial map between $MOR(E,F)$'s :
\begin{Defn} 
We define a simplicial map,
\begin{align}
\label{GCH_cat_ver_composition_simplicial}
MOR(E,F) \times MOR(F,G) \to MOR(E,G) .
\end{align}
It suffices to define a map which is compatible with face maps and degeneracy maps,
\begin{align}
Mor(E\wedge (\triangle (n) ) ^+,F) \times Mor(F\wedge (\triangle (n) ) ^+,G) &\to Mor(E\wedge (\triangle (n) ) ^+,G) \\
(~f~,~g~) &\mapsto h
\end{align}
We define $h$ by taking compositions of the following maps where $f \in Mor(E\wedge (\triangle (n) ) ^+,F)$, $g, \in Mor(F\wedge (\triangle (n) ) ^+,G)$ :
\begin{align}
E\wedge (\triangle (n) ) ^+
&\stackrel{\Delta}{\to} E\wedge (\triangle (n) \times \triangle (n)) ^+ \\
&\stackrel{\cong}{\to} (E\wedge (\triangle (n))^+) \wedge (\triangle (n))^+ \\
&\stackrel{f}{\to} F \wedge (\triangle (n))^+ \\
&\stackrel{g}{\to} G .
\end{align}

\end{Defn}

Since the composition simplicial map defined above satisfies the associativity strictly, it follows that (\ref{GCH_cat_ver_composition_simplicial}) is natural with respect to $E,F,G$.
For example, if we write by $\ast$ the one-point simplicial set, then $f \in Mor (E,F)$ induces a simplicial map $\hat{f} : \ast \to MOR (E,F)$.
Then we obtain a simplicial map from (\ref{GCH_cat_ver_composition_simplicial}) :
\begin{align}
MOR(F,G) &\to \ast \times MOR(F,G) \\
&\stackrel{\hat{f}}{\to} MOR(E,F) \times MOR(F,G) \\
&\to MOR(E,G)
\end{align}
If we denote it by $f^\ast : MOR(F,G) \to MOR(E,G)$, then we have $(g\circ f)^\ast = f^\ast \circ g^\ast$ due to the associativity of the composition simplicial map.
Therefore we obtain the following proposition.

\begin{prop}
The composition simplicial map induces a functor between fundamental groupoids,
\begin{align}
\Pi_1 MOR( E,F) \times \Pi_1 MOR(F,G) \to \Pi_1 MOR(E,G) .
\end{align}
This composition functor satisfies the associativity strictly, has a strict unit, and it is natural with respect to $E,F,G$.
\end{prop}

From now on, we discuss an additivity of $MOR (E,F)$ with respect to $E$.
Note that we have a bijection $Mor (E\vee E^\prime , F) \to Mor (E,F) \times Mor (E^\prime , F)$, which extends to a simplicial isomorphism $MOR (E\vee E^\prime , F) \to MOR (E,F) \times MOR (E^\prime , F)$.
Hence, we obtain the following proposition.

\begin{prop}
\label{first_variable_MOR_decomp}
Let $E,E^\prime,F$ be spectra.
Let us denote by $i : E \to E\vee E^\prime$, $i^\prime : E^\prime \to E \vee E^\prime$ the canonical inclusions.
They induce an isomorphism of groupoids :
\begin{align}
\label{spectra_wedge_isom}
\Pi_1 MOR (E \vee E^\prime , F) \to \Pi_1 MOR(E,F) \times \Pi_1 MOR(E^\prime ,F) .
\end{align}

\end{prop}

By definition, $\pi_0 (MOR(E,F))$ is the homotopy set of morphisms from $E$ to $F$.
It is well-known that the set of homotopy set of morphisms of spectra has an abelian group structure.
It is explained as follows.
Note that we have an equivalence $MOR (E,F) \simeq MOR(S\wedge E , F)$ where we denote by $S$ the sphere spectrum.
The sphere spectrum has a cogroup structure in the homotopy category of spectra with the wedge product as its monoidal structure :
We have a pinch map $S^1 \to S^1 \vee S^1$ which extends to a morphism of spectra $S\to S\vee S$.
Since $S^2 \to S^2 \vee S^2$ gives an abelian cogroup in the homotopy category of pointed spaces with the wedge product as its symmetric monoidal structure, we see that $S \to S \vee S$gives an abelian cogroup in the homotopy category of spectra.
Therefore the set $\pi_0 (MOR(S\wedge E , F))$ has an abelian group structure since we have $MOR((S\vee S)\wedge E , F) \cong MOR (S\wedge E, F) \times MOR (S\wedge E, F)$.
In a similar way, $\Pi_1 (S\wedge E , F)$ inherits a symmetric categorical group structure.

\begin{prop}
\label{E_vee_oplus}
By Proposition \ref{first_variable_MOR_decomp}, we have a symmetric monoidal isomorphism :
\begin{align}
\Pi_1 MOR (S \wedge ( E \vee E^\prime) , F) \to \Pi_1 MOR(S \wedge E,F) \oplus \Pi_1 MOR(S \wedge E,F) \notag
\end{align}
\end{prop}

Let $E,F,F^\prime$ be spectra.
Denote the canonical inclusions by $i : F \to F \vee F^\prime$, $i^\prime : F^\prime \to F \vee F^\prime$.
They induce a functor $f$ by
\begin{align}
\Pi_1 MOR (S \wedge E, F) \oplus \Pi_1 MOR (S \wedge E, F^\prime) &\to \Pi_1 MOR (S \wedge E, F \vee F^\prime) \notag \\
(a,b) &\mapsto i_\ast (a) \oplus i^\prime_\ast (b) \notag 
\end{align}
If $p : F \vee F^\prime \to F$ and $p^\prime : F \vee F^\prime \to F^\prime$ denotes the projections, we obtain a functor $g$, 
\begin{align}
\Pi_1 MOR (S \wedge E, F \vee F^\prime) &\to \Pi_1 MOR (S \wedge E, F) \oplus \Pi_1 MOR (S \wedge E, F^\prime) \notag \\
a &\mapsto (p_\ast (a) , p^\prime_\ast (a)) \notag
\end{align}

\begin{prop}
\label{F_vee_oplus}
The functors $f,g$ are lifted to an adjoint equivalence of symmetric categorical groups with canonical inverses.
\end{prop}
\begin{proof}
Since $p \circ i = id_{F}$, $p^\prime \circ i^\prime = id_{F^\prime}$ and $p \circ i^\prime$ and $p^\prime \circ i$ are collapsing morphisms, $g(f( a,b)) = f(i_\ast (a) \oplus i^\prime_\ast (b)) = (p_\ast (i_\ast (a) \oplus i^\prime_\ast (b)) , p^\prime_\ast (i_\ast (a) \oplus i^\prime_\ast (b))) = (a \oplus (p\circ i^\prime)_\ast (b) , (p^\prime \circ i)_\ast (a) \oplus b)\cong (a, b)$.
Hence we obtain a natural isomorphism $\eta : g\circ f \cong id$ which is natural with respect to $E,F,F^\prime$.

On the one hand, $f$ induces an equivalence of groupoids since its $\pi_0 , \pi_1$ induce isomorphisms 
Therefore, there exists a unique natural isomorphism $\epsilon : id \cong f \circ g$ such that $f,g, \eta , \epsilon$ gives an adjoint equivalence of groupoids.

The functor $g$ is obviously a (strict) symmetric monoidal functor preserving canonical inverses due to definitions.
Although it is not that obvious for the functor $f$, the functor $f$ is enhanced naturally to a symmetric monoidal functor preserving canonical inverses using the adjoint equivalence $f,g,\eta,\epsilon$.
Then $f,g,\epsilon,\eta$ give an adjoint equivalence of symmetric categorical groups with canonical inverses.
\end{proof}

\subsection{Construction}
\label{categorical cohomology group}

\begin{Defn}
Let $E$ be a spectrum.
For a pointed space $X$, we define a symmetric categorical group $\hat{E}^n (X)$ and $\hat{E}_n (X)$ by
\begin{align}
\hat{E}^n (X) &\stackrel{\mathrm{def.}}{=} \Pi_1 MOR (S \wedge \Sigma^{\infty} X , E)_{-n} , \\
\hat{E}_n (X) &\stackrel{\mathrm{def.}}{=} \Pi_1 MOR (S , E\wedge ( S \wedge \Sigma^{\infty} X ) )_n .
\end{align}
\end{Defn}

By definition, we have following isomorphisms of abelian groups.
\begin{align}
\pi_0 (\hat{E}^n (X)) \cong \tilde{E}^n (X) ,\\
\pi_0 (\hat{E}_n (X)) \cong \tilde{E}_n (X) .
\end{align}
Here, the right hand side denotes the reduced $E$-(co)homology groups.
Moreover, for the fundamental groups of groupoids $\hat{E}^n (X) , \hat{E}_n (X) $, we have the following proposition.

\begin{prop}
Let $E$ be a spectrum and $X$ be a pointed space.
We have following isomorphisms of abelian groups.
\begin{align}
\pi_1 (\hat{E}^n (X)) \cong \tilde{E}^{n-1} (X) , \\
\pi_1 (\hat{E}_n (X)) \cong \tilde{E}_{n+1} (X) .
\end{align}
\end{prop}
\begin{proof}
We sketch the proof for the first claim.
\begin{align}
\pi_1 (\hat{E}^n (X)) &\cong \pi_1 (MOR (S \wedge \Sigma^{\infty}X , E)_{-n}) \\
&\cong  Mor (S\wedge \Sigma^{\infty}X \wedge S^1 , E)_{-n} / homotopy \\
&\cong Mor (S\wedge \Sigma^{\infty}X , E)_{-n+1} / homotopy \\
&\cong \tilde{E}^{n-1} (X) 
\end{align}
\end{proof}

The symmetric categorical group version of generalized (co)homology theory defined above is motivated by \cite{SV}, \cite{RMV}.
Sharma-Voronov used a symmetric categorical group version of ordinary (co)homology theory to construct the Dijkgraaf-Witten theory.
Their categorical (co)homology group is based on \cite{RMV}, i.e. it is induced from a chain complex of categorical groups.
Then a universal element $u \in H^n (HM_n; M[0])$ induces an equivalence of symmetric monoidal categories :
\begin{align}
\hat{HM}^n (X ) = \Pi_1 MOR(S\wedge \Sigma^{\infty}X , HM)_n \to H^n (X ; M[0] ).
\end{align}
Here $HM$ be the Eilenberg-MacLane spectrum associated with an abelian group $M$, and $H^n (X; M[0])$ is the categorical group described in Sharma-Voronov \cite{SV}.
$M[0]$ is the discrete symmetric categorical group induced by the abelian group $M$.

\begin{prop}
\label{first_wedge_EE}
The inclusions $X \to X \vee X^\prime$ and $X \to X \vee X^\prime$ induce the following symmetric monoidal isomorphism :
\begin{align}
\hat{E}^n (X \vee X^\prime ) \to \hat{E}^n (X) \oplus \hat{E}^n (X^\prime ) .
\end{align}
\end{prop}
\begin{proof}
It follows from Proposition \ref{E_vee_oplus}.
\end{proof}

\begin{prop}
\label{second_wedge_EE}
We have a symmetric monoidal adjoint equivalence which is natural with respect to $X,Y,Y^\prime$ :
\begin{align}
 \hat{E}_n (X \vee X^\prime ) \to \hat{E}_n (X) \oplus \hat{E}_n (X^\prime ) . 
\end{align}
\end{prop}
\begin{proof}
It follows from Proposition \ref{F_vee_oplus}.
\end{proof}

Consider a pairing of spectra $\mu : E \wedge F \to G \wedge S$ where $E,F,G$ are spectra.
The pairing $\mu$ induces a pairing of symmetric categorical groups $\hat{F}^n (X) \times \hat{E}_m (X) \to \hat{G}_{m-n} (S^0)$ as a functor.
It is constructed by composing the following simplicial maps.
\begin{align}
MOR( S \wedge \Sigma^{\infty} X , F)_{-n} \times MOR( S , E\wedge ( S \wedge \Sigma^{\infty} X) )_m 
&\to
MOR (S , E \wedge F)_{m-n} \\
&\to 
MOR (S, G \wedge S)_{m-n} = \hat{G}_{m-n} (S^0)
\end{align}

\begin{remark}
\label{rel2exactsequ_remark}
For a pointed continuous map $f : X \to Y$, there is an associated long cofiber sequence $X \to Y \to C_f \to \Sigma X \to \Sigma Y \to \Sigma C_f \to \cdots$, which is called Puppe sequence.
If we applying generalized (co)homology theory to this sequence, then we obtain a long exact sequence which may be used to prove well-known long exact sequences.

On the one hand, there is lifted versions of long exact sequences given by Theroem 4.2 \cite{RMV}.
They are called 2-exactness and relative 2-exactness where 2-exactness implies relative 2-exactness.
We have a relative 2-exact sequence of our categorical (co)homology groups associated with Puppe sequence.
Let $f : X \to Y$ be a based map.
The long cofiber sequence $X \to Y \to C_f \to \Sigma X \to \Sigma Y \to \Sigma C_f \to \cdots$ associated with $f$ induces a relative 2-exact sequence : 
$$
\begin{tikzcd}
\hat{E}_n (X) \ar[r] \ar[rr, bend left, ""{name =A}, "0"] & \hat{E}_n  (Y)  \ar[r] \ar[rr, bend right, ""{name=B}, "0"'] \ar[to = A, Rightarrow] & \hat{E}_n (C_f) \ar[r]  \ar[rr, bend left, ""{name=C}, "0"] \ar[to = B , Rightarrow] & \hat{E}_n  (\Sigma X) \ar[r] \ar[to = C , Rightarrow] & \cdots \\
\end{tikzcd}
$$
$$
\begin{tikzcd}
\hat{E}^n (X)
\ar[rr, bend left, leftarrow, ""{name =A}, "0"]
&
\hat{E}^n (Y)
\ar[l]
\ar[to = A ,Rightarrow]
\ar[rr, bend right, leftarrow, ""{name =B}, "0"']
&
\hat{E}^n (C_f)
\ar[l]
\ar[to = B ,Rightarrow]
\ar[rr, bend left, leftarrow, ""{name =C}, "0"]
&
\hat{E}^n (\Sigma X)
\ar[to = C ,Rightarrow]
\ar[l]
&
\cdots
\ar[l]
\end{tikzcd}
$$
Its proof essentially comes from the (usual) exact sequence of generalized (co)homology groups.
\end{remark}


\section{Classical theory from generalized (co)homology}
\label{Classical theory from generalized (co)homology}

In this section, we construct a Lagrangian classical field theory $\hat{S}_{\hat{\varphi},\hat{\alpha}}$ for $E$-oriented manifolds as a symmetric monoidal functor where $E$ is a ring spectrum.
It is possible to regard  $\hat{S}_{\hat{\varphi},\hat{\alpha}}$ as an extension of the invariant $S_{\varphi, \alpha}$ of $B\Gamma$-manifolds in Definition \ref{CFT_GCH_Salpha_defn}.
The construction is, in fact, applied to arbitrary space instead of $B\Gamma$.

\begin{Defn}
Consider a diagram of spaces $R$,
\begin{align}
R_0 \to R_{01} \leftarrow R_1 .
\end{align}
Such a diagram is called a cospan, but we call it as a 1-diagram of spaces in this paper.
We write by $\partial_0 (R) \stackrel{\mathrm{def.}}{=} R_1$ and $\partial_1 (R) \stackrel{\mathrm{def.}}{=} R_0$.

Let us consider the following commutative diagram of spaces $R$,
$$
\begin{tikzcd}
R_{0} \arrow[r] \arrow[d] & R_{01} \arrow[d] & R_{1} \arrow[l] \arrow[d] \\
R_{02} \arrow[r] & R_{012} & R_{12} \arrow[l] \\
 & R_{2} \arrow[lu] \arrow[ru] & 
\end{tikzcd}
$$
Let us call such a diagram as a 2-diagram of spaces.
We define 1-diagram $\partial_k (R)$ for $k=0,1,2$ as follows.
\begin{itemize}
\item
$\partial_0 (R) \stackrel{\mathrm{def.}}{=} R_1 \to R_{12} \leftarrow R_2$ .
\item
$\partial_1 (R) \stackrel{\mathrm{def.}}{=} R_0 \to R_{02} \leftarrow R_2$ .
\item
$\partial_2 (R) \stackrel{\mathrm{def.}}{=} R_0 \to R_{01} \leftarrow R_1$ .
\end{itemize}

We call a diagram consisting of one space by 0-diagram.
For a 0-diagram $R$, we define $A_0 (R) \stackrel{\mathrm{def.}}{=} D^2 \times R$.

Let $R = R_0 \stackrel{i}{\to} R_{01} \stackrel{j}{\leftarrow} R_1$ be a 1-diagram as above.
We use the maps $i,j$ in the diagram to glue the following disjoint union of
\begin{align}
\amalg_k \left( R_k \times \{ r e^{ \sqrt{-1} \pi t} \in D^2 ~|~ k \leq t \leq k+1 \} \right) \amalg \left( R_{01} \times \{ -1 \leq r \leq 1, ~ r \in D^2 \} \right) .
\end{align}
Here, we regard the 2-disk embedded in $\mathbb{C}$.
We denote the space by $A_1 (R)$.

Let $R$ be a 2-diagram as above.
We glue following spaces using the maps associated with the diagram $R$.
\begin{itemize}
\item
$\amalg_k \left( R_k \times \{ r e^{2\pi t/3} ~|~ k \leq t \leq k+1 \} \right)$
\item
$ R_{01} \times \{ r \in D^2 \} \amalg R_{12} \times \{r e^{2\pi /3} \in D^2 \} \amalg R_{02} \times \{ r e^{4\pi /3} \in D^2 \} $
\item
$R_{012} \times \{ 0 \}$
\end{itemize}
Here, $r$ denotes a non-negative real number.
We detnoe by the space $A_2 (R)$.
\end{Defn}

\begin{remark}
We have an obvious projection $A_k (R) \to D^2$ for a $k$-diagram $R$.
We consider $A_k (R)$ as a bundle over $D^2$ via the projection
\end{remark}

\begin{Defn}
\label{defn_2disk_B}
We define a pointed space $B_k (R)$ for a k-diagram $R$ by
\begin{align}
B_k (R) \stackrel{\mathrm{def.}}{=} A_k (R) / A_k (R) |_{\partial D^2}
\end{align}
For a morphism between $k$-diagrams $f : R \to R^\prime$, we denote by $B_k (f) : B_k (R) \to B_k (R^\prime)$ the induced continuous map.

Fix orientation preserving embeddings $\theta_{1,k} : D^2 \to D^2$ for $k=0,1$ such that
\begin{align}
D^2 \backslash \theta_{1,k} (D^2) = \{ r e^{\sqrt{-1}\pi t} \in D^2 ~|~ k < t < k+1,~r \neq 0 \} .
\end{align}
Then $\theta_{1,k}$ is lifted to an embedding $\bar{\theta}_{1,k} : A_0 (\partial_k R) \to A_1 (R)$ for a 1-diagram $R$.
It induces a pointed map $q_{1,k} : B_1 (R) \to B_0 (\partial_k R)$ for $k=0,1$.

In a similar way, we define $q_{2,k} : B_2 (R) \to B_1 (\partial_k R)$ where $k= 0,1,2$ and $R$ is a 2-diagram of spaces.
Fix orientation preserving embeddings $\theta_{2,k} : D^2 \to D^2$ for $k=0,1,2$ such that
\begin{align}
D^2 \backslash \theta_{2,k} (D^2) =  \{ r e^{\sqrt{-1} 2\pi t /3} \in D^2 ~|~ k < t < k+1,~r \neq 0 \} .
\end{align}
It is lifted to an embedding $\bar{\theta}_{2,k} : A_1 (\partial_k R) \to A_2 (R)$ for a 2-diagram $R$, which induces a pointed map $q_{2,k} : B_2 (R) \to B_1 (\partial_k R)$.
\end{Defn}

\begin{Defn}
Let $L$ be a compact $n$-manifold possibly with boundary.
We define a category $\mathcal{C} (L)$ of embeddings from $n$-disk $D^n$ into $L$.
For two embeddings $i_0,i_1 : D^n \to L$, a pre-morphism $j$ is given as an embedding $j: D^n \times [0,1] \to L \times [0,1]$ such that
\begin{itemize}
\item
$j(x,t) = (j_t (x) , t)$
\item
$j(x,0) = i_0 (x)$ and $j(x,1) =i_1 (x)$
\end{itemize}
We identify two pre-morphisms $j_0 , j_1$ if there exists an embedding $h : (D^n \times [0,1] ) \times [0,1] \to (L\times [0,1] ) \times [0,1]$ such that
\begin{itemize}
\item
$h(x,t,0) = j_0 (x,t)$ and $h(x,t,1) = j_1 (x,t)$
\item
$h(x,0,s) = i_0 (x)$ and $h(x,1,s) = i_1 (x)$
\end{itemize}
We define the composition $j^\prime \sharp j$ of pre-morphisms $j: i_0 \to i_1$ and $j^\prime : i_1 \to i_2$ as
\begin{align}
(j^\prime \sharp j) (x,t) \stackrel{\mathrm{def.}}{=} 
\begin{cases}
(j_{2t} (x) , t) &\text{if $0 \leq t \leq 1/2$} \\
(j^\prime_{2t-1} (x) , t) &\text{if $1/2 t \leq 1$.}
\end{cases}
\end{align}

Let us define a functor $\chi_L : \mathcal{C} (L) \to \Pi_1 Map (L/\partial L , D^n / \partial D^n)$.
For an object $i$ of $\mathcal{C} (L)$, i.e. an embedding $i: D^n \to L$, we define a map $\chi_L (i) \in Map (L/\partial L , D^n/ \partial D^n) $ as the collapsing map $L / \partial L \to i(D^n)/i(\partial D^n) \cong D^n / \partial D^n$.
For a morphism $j$ from $i_0$ to $i_1$, we define a homotopy $\chi_L (j) =h$ from $\chi_L (i_0)$ to $\chi_L (i_1)$ by
\begin{align}
h (x, t) \stackrel{\mathrm{def.}}{=} 
\begin{cases}
[y] &\text{if $j(y,t) = (x,t) \in L\times [0,1]$} \\
\ast &\text{otherwise}
\end{cases}
\end{align}
\end{Defn}

\begin{remark}
Suppose that a $1$-diagram $R = ( R_0 \to R_{01} \leftarrow R_1)$ is constant.
Then $q_{1,k} : B_1 (R) \to B_0 (\partial_k R)$ coincides with $\chi_L (\theta_{1,k}) \wedge id_{R^+_{01}}$.
Similar claim holds for $2$-diagrams.
\end{remark}

\begin{Defn}
Let us define a symmetric monoidal category $\mathbf{Cob}^{E,rep}_{n-1} (T)$ for a space $T$.

A triple $(Y, g, \xi)$ is an object of $\mathbf{Cob}^{E,rep}_{n-1} (T)$ if the followings hold :
\begin{itemize}
\item
$Y$ is a closed $E$-oriented $(n-1)$-manifold such that $E_{n} (Y) \cong 0$.
\item
$g : Y \to T$ is a continuous map.
\item
$\xi \in \hat{E}_{n+1} (B_0 (Y))$ which represents the $E$-orientation of $Y$.
\end{itemize}

We define a morphism as an equivalence class of pre-morphisms which we explain from now on.
A quintuple $(X ,Y_0 , Y_1 , f , \eta)$ is a pre-morphism from $(Y_0, g_0, \xi_0)$ to $(Y_1, g_1, \xi_1)$ if the followings hold :
\begin{itemize}
\item
$X$ is an $E$-oriented bordism from $Y_0$ to $Y_1$ which are closed $E$-oriented manifolds such that $E_n (Y_0) \cong 0 \cong E_n (Y_1)$.
\item
$f: X \to T$ is a continuous map whose restriction to boundaries coincides with the map $g_0 \amalg g_1 : Y_0 \amalg Y_1 \to T$.
\item
$\eta \in \hat{E}_{n+1} (B_1 (X) )$ represents the $E$-orientation of $X$.
Here, $B_1 (X)$ denotes $B_1 (R)$ where $R$ is the 1-diagram of spaces $Y_0 \to X \leftarrow Y_1$ where the maps are inclusions.
\end{itemize}

We define the identity in the category as follows.
For an object $(Y, g, \xi)$, we define its identity as the class $[ Y\times [0,1] , g \circ \pi , \xi^\prime]$ where $\pi : Y \times [0,1] \to Y$ is the projection and $\xi^\prime$ is a representative of the $E$-orientation of $Y \times [0,1]$.

We define the composition in the category as follows.
For two pre-morphisms $(X ,Y_0 , Y_1 , f , \eta)$, $(X^\prime ,Y_1 , Y_2 , f^\prime, \eta^\prime )$, we define their composition as the class of a pre-morphism $(X^\prime \circ X , Y_0 , Y_2 , f^\prime \sharp f , \eta^{\prime\prime} )$ where $f^\prime \sharp f$ is the map obtained by gluing $f,f^\prime$ and $ \eta^{\prime\prime}$ is a representative of $E$-orientation of $X^\prime \circ X$.
\end{Defn}

Then the data of $\mathbf{Cob}^{E,rep}_{n-1} (T)$ defined above forms a category.
We construct a symmetric monoidal category structure on the category $\mathbf{Cob}^{E,rep}_{n-1} (T)$.
We define its `tensor product' by
\begin{align}
(Y, g, \xi) \amalg (Y^\prime , g^\prime , \xi^\prime ) \stackrel{\mathrm{def.}}{=}  (Y \amalg Y^\prime , g \amalg g^\prime , \xi \oplus \xi^\prime )
\end{align}
Then there are obvious associator and left (right) unit morphisms.
Note that the associator is induced by the associator of the monoidal structure of $\hat{E}_{n+1} ( Y \amalg Y^\prime \amalg Y^{\prime\prime})$.

\begin{Defn}
Let us define a symmetric monoidal functor $\Phi_T : \mathbf{Cob}^{E,rep}_{n-1} (T) \to \hat{E}_{n+1} ( B_0 (T) )$.
For simplicity, we denote by $f_\ast = \hat{E}_{n+1} (f)$ for a continuous map $f$ and $h_{\ast\ast} = \hat{E}_{n+1} (h)$ for a homotopy $h$.

Let $(Y,g,\xi)$ be an object of $\mathbf{Cob}^{E,rep}_{n-1} (T)$.
We define $\Phi_T (Y,g,\xi)$ by
\begin{align}
\Phi_T (Y,g,\xi) \stackrel{\mathrm{def.}}{=} B_0 (g)_\ast \xi \in \hat{E}_{n+1} (B_0 (B\Gamma)) .
\end{align}
Let $[X ,Y_0 , Y_1 , f , \eta]$ be a morphism from $(Y_0,g_0,\xi_0)$ to $(Y_1,g_1,\xi_1)$ in the category $\mathbf{Cob}^{E,rep}_{n-1} (T)$.
In order to define $\Phi_T (X ,Y_0 , Y_1 , f , \eta)$, we choose and fix a morphism $\gamma : \theta_{1,1} \to \theta_{1,0}$ in the category $\mathcal{C}(L), ~L=D^2$.
The bordism $X$ induces a 1-diagram $R= (R_0 \to R_{01} \leftarrow R_1)=(Y_0 \stackrel{i_1}{\to} X \stackrel{i_0}{\leftarrow} Y_1)$.
Then we are given following diagram.
Note that we have $B_0 (i_k) \circ q_{1,k} = \chi_L (\theta_{1,k}) \wedge id_{R^+_{01}}$.
Hence, the morphism $\gamma$ induces a homotopy $h_X : B_0 (i_1) \circ q_{1,1} \to B_0 (i_0) \circ q_{1,0} $ upto homotopy, i.e. $h_X$ is a morphism in $\Pi_1 Map (B_1 (R) , B_0 (R_{01}))$.
Note that we have unique morphism $\kappa_1 : (q_{1,1})_\ast \eta \to \xi_0$ and $\kappa_0 : (q_{1,0})_\ast \eta \to \xi_1$ since $\hat{E}_{n+1} ( B_0 (Y_k))$'s are simply connected by assumptions.
We define 
\begin{align}
\Phi_T [X ,Y_0 , Y_1 , f , \eta] \stackrel{\mathrm{def.}}{=} \left( B_0(g_1)_\ast \kappa_0 \right) \circ \left( B_0(f)_\ast (h_X)_{\ast\ast} \eta \right) \circ  \left( B_0(g_0)_\ast \kappa_1 \right),
\end{align}
which is a morphism from $B_0 (g_0)_\ast \xi_0$ to $B_0 (g_1)_\ast \xi_1$ in the category $\hat{E}_{n+1}(B_0 (B\Gamma))$.

We define a natural morphism $\Phi_T ((Y,g,\xi) \amalg (Y^\prime ,g^\prime \xi^\prime)) \to \Phi_T (Y,g,\xi) \oplus \Phi_T (Y^\prime ,g^\prime ,\xi^\prime)$ as the morphism $B_0 (g\amalg g^\prime)_\ast (\xi \oplus \xi^\prime) \to B_0 (g)_\ast (\xi) \oplus B_0 (g^\prime)_\ast (\xi^\prime)$ induces by the monoidality of the monoidal functor $B_0$ and monoidal category $\hat{E}_{n+1}(B\Gamma)$.
\end{Defn}

\begin{prop}
\label{phi_is_smf}
$\Phi_T$ gives a symmetric monoidal functor.
\end{prop}
\begin{proof}
We prove that $\Phi_T$ sends an identity morphism to an identity morphism.
Let $[Y\times [0,1] , g \circ \pi , \xi^\prime]$ be an identity morphism on the object $(Y, g , \xi) \in \mathbf{Cob}^{E,rep}_{n-1} (T) $.
We have,
\begin{align}
\Phi_{T} ([Y\times [0,1] , g \circ \pi , \xi^\prime]) &= B_0 (g \circ \pi)_\ast (h_{Y\times [0,1]})_{\ast\ast} \eta \\ 
&= B_0 (g)_\ast \circ B_0 (\pi)_\ast (h_{Y\times [0,1]})_{\ast\ast} \eta .
\end{align} 
Note that the target category of the functor $B_0 (\pi)_\ast$ is $\hat{E}_{n+1} (B_0 (Y))$ which is simply-connected by the assumption $E_n (Y) \cong 0$.
Hence, $B_0 (\pi)_\ast (h_{Y\times [0,1]})_{\ast\ast} \eta $ is the identity morphism on $B_0 (g)_\ast \xi$ so that $\Phi_T$ sends an identity morphism to an identity morphism.

We prove that $\Phi_T$ preserves the composition of morphisms.
Let us consider a 2-diagram $R$ such that $R_{012}= X^\prime \circ X$, $\partial_2 R = ( Y_0 \to X \leftarrow Y_1)$, $\partial_0 R = (Y_1 \to X^\prime \leftarrow Y_2)$, $\partial_1 R = (Y_0 \to X^\prime \circ X \leftarrow Y_2)$.
We set the structure map $(R_{02} \to R_{012} ) = (X^\prime \circ X \to X^\prime \circ X$ to be the identity map, $(R_{12} \to R_{012}) = (X^\prime \to X^\prime\circ X$ and $(R_{01} \to R_{012}) = (X \to X^\prime \circ X)$ to be inclusions.
We choose isotopies of embeddings $k^\prime_0,k^\prime_1,k^\prime_2$ given in the diagram below :
\begin{equation}
\begin{tikzcd}
 & A_2 (R) &  \\
A_1 (\partial_a R) \arrow[ru, "\bar{\theta}_{2,a}"] & \stackrel{k^\prime_{c}}{\Longrightarrow} & A_1 (\partial_b R) \arrow[lu, "\bar{\theta}_{2,b}"] \\
 & A_0 (R_c)\arrow[lu, "\bar{\theta}"] \arrow[ru, "\bar{\theta}"]  & 
\end{tikzcd}
\end{equation}
Here, $a < b$ and $\{ a,b,c\} = \{0,1,2 \}$.
Then the isotopy $k^\prime_c$ induce a homotopy $k_c$ :
\begin{equation}
\begin{tikzcd}
 & B_2 (R) \arrow[ld] \arrow[rd] &  \\
B_1 (\partial_a R) \arrow[rd] & \stackrel{k_c}{\Longrightarrow} & B_1 (\partial_b R) \arrow[ld] \\
 & B_0 (R_c) & 
\end{tikzcd}
\end{equation}
They can be denoted as follows.
\begin{itemize}
\item $k_0 :  q_{1,1}\circ q_{2,1} \to q_{1,1} \circ q_{2,2}$
\item $k_1 : q_{1,1} \circ q_{2,0} \to  q_{1,0}\circ q_{2,2}$
\item $k_2:  q_{1,0}\circ q_{2,0} \to q_{1,0} \circ  q_{2,1}$
\end{itemize}
We can choose such $k_c$'s satisfying the following commutative diagrams in $\Pi_1 Map ( B_2 (R) , B_0 (R_{012}))$.
\begin{equation}
\label{homotop_comm_require}
\begin{tikzcd}
B_0(i_0) \circ q_{1,1} \circ q_{2,2} \arrow[d, "h_{X}"] & B_0 (i_0) \circ q_{1,1} \circ q_{2,1} \arrow[l, "k_0"] \arrow[r, "h_{X^\prime\circ X}"] & B_0(i_2) \circ q_{1,0} \circ q_{2,1} \\
B_0 (i_1)\circ q_{1,0} \circ q_{2,2} &  & B_0 (i_2) \circ q_{1,0} \circ q_{2,0} \arrow[u , "k_2"] \\
 & B_0 (i_1) \circ q_{1,1} \circ q_{2,0} \arrow[lu, "k_1"] \arrow[ru, "h_{X^\prime}"] & 
\end{tikzcd}
\end{equation}
Here, we denote by $i_k : Y_k \to X^\prime \circ X$ the inclusion for $k=0,1,2$.
\begin{align}
\Phi_T & (X^\prime , Y_1 , Y_2 , f^\prime , \eta^\prime) \circ \Phi_T (X , Y_0 , Y_1 , f , \eta) \\
&= \Phi_T (X^\prime , Y_1 , Y_2 , f^\prime , q_{2,0}\omega ) \circ \Phi_T (X , Y_0 , Y_1 , f , q_{2,2} \omega ) \\
&= \left( B_0 (g_2)_\ast \kappa^\prime_2 \right) \circ \left( B_0 (f^\prime)_\ast (h_{X^\prime})_{\ast\ast} q_{2,0} \omega \right) \circ \left( B_0 (g_1)_\ast \kappa^\prime_1 \right) \\
& \circ \left( B_0 (g_1)_\ast \kappa_1 \right) \circ \left( B_0 (f)_\ast (h_{X})_{\ast\ast} q_{2,2} \omega \right) \circ  \left( B_0 (g_0)_\ast \kappa_0 \right)
\end{align}
Note that $\kappa^\prime_1 \circ \kappa$ is the unique morphism connecting its source and target.
Hence, we have $\kappa^\prime_1 \circ \kappa = (k_1)_{\ast\ast} \omega$.
\begin{align}
&\Phi_T (X^\prime \circ X , Y_0 , Y_2 , f^\prime\sharp f , \eta^{\prime\prime}) \\
&= \Phi_T (X^\prime \circ X, Y_0 , Y_2 , f^\prime \sharp f , q_{2,1} \omega ) \\
&= \left( B_0 (g_2)_\ast \kappa^\prime_2 \right) \circ \left( B_0 (f^\prime \sharp f)_\ast (h_{X^\prime\circ X})_{\ast\ast} q_{2,1} \omega \right) \circ \left( B_0 (g_0)_\ast \kappa_0 \right) 
\end{align}
Above all, we can check that $\Phi_T (X^\prime \circ X , Y_0 , Y_2 , f^\prime\sharp f , \eta^{\prime\prime}) = \Phi_T (X^\prime , Y_1 , Y_2 , f^\prime , \eta^\prime) \circ \Phi_T (X , Y_0 , Y_1 , f , \eta) $ due to the uniqueness of $\kappa_a , \kappa^\prime_a$'s and the commutative diagram (\ref{homotop_comm_require}).

The symmetric monoidality of the functor $\Phi_T$ is shown using the monoidal structure of category $\hat{E}_{n+1}(B\Gamma)$, functor $B_0 (-)$, and the symmetric monoidal equivalence in Proposition \ref{second_wedge_EE}.
\end{proof}

\begin{Defn}
Let $\hat{\alpha} \in \hat{F}^{n+2} ( B_0 (B\Gamma))$ be an object.
It induces a symmetric monoidal functor $\langle \hat{\alpha} , - \rangle : \hat{E}_{n+1} (B_0 (B\Gamma)) \to \hat{G}_{-1}(S^0)$.
For a symmetric monoidal functor $\hat{\varphi} : \hat{G}_{-1}(S^0) \to \mathbb{F}^{\times}\mathbf{Tor}$, we define a symmetric monoidal functor $\hat{S}_{\hat{\varphi},\hat{\alpha}} : \mathbf{Cob}^{E}_{n-1} (B\Gamma) \to \mathbb{F}^{\times}\mathbf{Tor} $ by composing following symmetric monoidal functors :
\begin{align}
\mathbf{Cob}^{E,rep}_{n-1} (B\Gamma) \stackrel{\Phi_{B\Gamma}}{\to} \hat{E}_{n+1} (B_0 (B\Gamma)) \stackrel{\langle \hat{\alpha},-\rangle}{\to} \hat{G}_{-1} (S^0) \stackrel{\hat{\varphi}}{\to} \mathbb{F}^{\times}\mathbf{Tor} 
\end{align}
, and using the equivalence $\mathbf{Cob}^{E,rep}_{n-1} (B\Gamma) \to \mathbf{Cob}^{E}_{n-1} (B\Gamma) ; (Y,g,\xi) \mapsto (Y,g)$.
\end{Defn}

\begin{remark}
Let $\hat{\varphi} : \hat{G}_{-1}(S^0) \to \mathbb{F}^{\times}\mathbf{Tor}$ be a symmetric monoidal functor.
Then $\hat{\alpha} \mapsto \hat{S}_{\hat{\varphi},\hat{\alpha}}$ gives a symmetric monoidal functor $\hat{F}^{n+2}(B_0 (B\Gamma)) \to \mathbf{Cob}^{n-1}_E ( B\Gamma ; \mathbb{F}^{\times}\mathbf{Tor})$.
\end{remark}


\section{Quantum theory from generalized (co)homology}
\label{Quantum theory from generalized (co)homology}

In this section, we construct a TQFT for $E$-oriented manifolds starting from a representative $\hat{\alpha}$ of a class $\alpha \in F^n (B\Gamma)$.
It is obtained by applying a push-forward to our classical Lagrangian classical field theory.
The TQFT yields an invariant for closed $E$-oriented $n$-manifolds which coincides with the generalized Dijkgraaf-Witten invariant defined in subsection \ref{subsec_DW_invariant_GCH}.

There exists a push-forward functor $\mathbf{Cob}^{n-1}_E (B\Gamma ;\mathbf{Vect}_{\mathbb{F}}) \to \mathbf{Cob}^{n-1}_E (\mathrm{pt} ;\mathbf{Vect}_{\mathbb{F}}) $.
It is constructed by using the ambidexterity of the category $\mathbf{Vect}_{\mathbb{F}}$ \cite{HL}.
We sketch the construction.
Let $\mathcal{C},\mathcal{D}$ be categories and $F : \mathcal{C} \to \mathcal{D}$ be a functor.
We suppose following assumptions.
\begin{itemize}
\item
The categories $F^{-1}(d)$, $F^{-1}(l)$ are finite groupoids for an object $d$ and a morphism $l$ in $\mathcal{D}$.
\item
Let $l: d_0 \to d_1$, $l^\prime : d_1 \to d_2$ be morphisms in $\mathcal{D}$.
Suppose that the canonical functor $F^{-1}(l)_p \times_q F^{-1}(l^\prime) \to f^{-1} (l^\prime \circ l)$ induces a equivalence of categories where $p: F^{-1}(l) \to F^{-1}(d_1)$, $q : F^{-1} (l) \to F^{-1}(d_1)$ are projections.
\end{itemize}
For an object $d \in \mathcal{D}$, we denote by $f^{-1}(d)$ the homotopy fiber category, i.e. the category of pair $(c, u)$ such that $c \in \mathcal{C}$ and $u : F(c) \to d$ is an isomorphism.
For a morphism $l : d_0 \to d_1$ in $\mathcal{D}$, we also denote by $F^{-1}(l)$ the homotopy fiber category, i.e. the category of triples $(h,u,v)$ such that $h$ is a morphism in $\mathcal{C}$, $u \circ F(h) \circ v = l$ and $u,v$ are isomorphisms in the category $\mathcal{D}$.
Under the assumptions, we construct a functor $F_{!} : Fun ( \mathcal{C} , \mathbf{Vect}_{\mathbb{F}} ) \to Fun ( \mathcal{D} , \mathbf{Vect}_{\mathbb{F}} )$ as follows.
Let $f \in Fun ( \mathcal{C} , \mathbf{Vect}_{\mathbb{F}} )$, i.e. $f$ is a functor from $\mathcal{C}$ to $\mathbf{Vect}_{\mathbb{F}}$.
\begin{itemize}
\item
$\left( F_{!} (f) \right) (d) \stackrel{\mathrm{def.}}{=} \varprojlim f|_{F^{-1}(d)}$ for $d \in \mathcal{D}$.
\item
Let $l : d_0 \to d_1$ be a morphism in $\mathcal{D}$.
We define a morphism $\left( F_{!} (f) \right) (l)$ from $\left( F_{!} (f) \right) (d_0)$ to $\left( F_{!} (f) \right) (d_1)$ by composing the following maps,
\begin{equation}
\label{ambidex_push_forward_dia}
\begin{tikzcd}
 & \varprojlim f\circ \pi_0 \stackrel{\cong}{\to} \varinjlim f\circ \pi_1 \arrow[ld] \arrow[rd] &  \\
\varprojlim f|_{F^{-1}(d_0)} &  & \varinjlim f|_{F^{-1}(d_1)} \stackrel{\cong}{\leftarrow} \varprojlim f|_{F^{-1}(d_1)}
\end{tikzcd}
\end{equation}
where $\pi_k : F^{-1}(l) \to F^{-1}(d_k)$ is the projection for $k=0,1$.
Here, we use the first assumption above and the ambidexterity of the category $\mathbf{Vect}_{\mathbb{F}}$ .
\end{itemize}
Then the assignment $F_! f$ defined previously gives a functor from $\mathcal{D}$ to $\mathbf{Vect}_{\mathbb{F}}$ due to the second assumption above.
Note that for the functor $c_\ast : \mathbf{Cob}^E_{n-1}(B\Gamma) \to \mathbf{Cob}^E_{n-1} (\mathrm{pt})$, we have equivalences $(c_\ast)^{-1}(Y) \simeq \Pi_1 Map (Y, B\Gamma)$ and $(c_\ast)^{-1}(X) \simeq \Pi_1 Map (X , B\Gamma)$ where $Y,X$ are an object and a morphism in the category $\mathbf{Cob}^E_{n-1} (\mathrm{pt})$ respectively.
Hence, $c_\ast$ satisfies the first assumption above.
Moreover, it satisfies the second assumption since the following diagram forms a homotopy pull-back diagram induced by restrictions where $X : Y_0 \to Y_1$, $X^\prime : Y_1 \to Y_2$ are bordisms.
$$
\begin{tikzcd}
 &  \Pi_1 Map (X^\prime \circ X, B\Gamma) \arrow[ld] \arrow[rd] &  \\
 \Pi_1 Map (X, B\Gamma) \arrow[rd] &  &  \Pi_1 Map (X^\prime, B\Gamma) \arrow[ld] \\
 &  \Pi_1 Map (Y_1, B\Gamma) & 
\end{tikzcd}
$$

\begin{Defn}
We define $\hat{Z}_{\hat{\varphi},\hat{\alpha}} \in \mathbf{Cob}^{n-1}_E (\mathrm{pt} ; \mathbf{Vect}_{\mathbb{F}})$ by 
\begin{align}
\hat{Z}_{\hat{\varphi},\hat{\alpha}} \stackrel{\mathrm{def.}}{=} c_{!} (\hat{S}_{\hat{\varphi},\hat{\alpha}}) .
\end{align}
\end{Defn}

\begin{theorem}
\label{corollary_from_cob_tqft_GCH}
Let $\hat{\varphi} : \hat{G}_{-1}(S^0) \to \mathbb{F}^{\times} \mathbf{Tor}$ be a symmetric monoidal functor.
Let $\varphi$ be the homomorphism $ \pi_1 (\hat{\varphi}) : G_0 (\mathrm{pt}) \cong \pi_1 (\hat{G}_{-1}(S^0)) \to \pi_1 ( \mathbb{F}^{\times} \mathbf{Tor}) \cong \mathbb{F}^{\times}$.
The invariant $Z_{\varphi,\alpha}$ in Definition \ref{TQFT_GCH_Zalpha_defn} extends to a strong symmetric monoidal functor :
\begin{align}
\hat{Z}_{\hat{\varphi},\hat{\alpha}} : ( \mathbf{Cob}^E_{n-1} , \amalg ) \to ( \mathbf{Vect}_{\mathbb{F}} , \otimes ) .
\end{align}
More precisely, we have
\begin{align}
inv : \pi_0 ( \mathbf{Cob}^{n-1}_E ( \mathrm{pt}  ; \mathbf{Vect}_{\mathbb{F}} ) ) &\to \mathbf{M}^n_E ( \mathrm{pt} ; k) , \\
[ \hat{Z}_{\hat{\varphi},\hat{\alpha}}  ] &\mapsto Z_{\varphi, \alpha} .
\end{align}
\end{theorem}
\begin{proof}
We consider the previous discussion for $\mathcal{C} = \mathbf{Cob}^E_{n-1}(B\Gamma), \mathcal{D} = \mathbf{Cob}^E_{n-1}(\mathrm{pt}) = \mathbf{Cob}^E_{n-1}$, $F=c_\ast$.
For the object $\emptyset  \in \mathbf{Cob}^E_{n-1}$ (the null space), we have $F^{-1} (\emptyset ) \simeq \ast$ so that $\varprojlim f|_{F^{-1}(\emptyset)} \cong \mathbb{F}$ as vector spaces.
Let $(X, \emptyset , \emptyset )$ be a morphism between null spaces in the category $\mathbf{Cob}^E_{n-1}$.
Note that we have $F^{-1} (X, \emptyset , \emptyset ) \simeq \Pi_1 Map (X, B\Gamma)$.
Hence, to compute $inv [\hat{Z}_{\hat{\varphi},\hat{\alpha}}]$, we check where $1 \in \mathbb{F} \cong \hat{Z}_{\hat{\varphi},\hat{\alpha}}(\emptyset)$ is sent by $\hat{Z}_{\hat{\varphi},\hat{\alpha}}(X)$ where $X$ is an $E$-oriented $n$-bordism which is closed.
Recall the definition of corresponding morphism in (\ref{ambidex_push_forward_dia}), then we obtain the followings.
$$
\begin{tikzcd}
 & (1 ~|~ f: X\to B\Gamma) \mapsto (\frac{S_{\varphi,\alpha}(X,f)}{\sharp Aut(f^\ast E\Gamma)}~|~f : X \to B\Gamma) \arrow[rd] &  \\
1 \in \mathbb{F} \arrow[ur] &  & \sum_{f} \frac{S_{\varphi,\alpha}(X,f)}{\sharp Aut(f^\ast E\Gamma)}
\end{tikzcd}
$$
It proves the claim since $\sum_{f} \frac{S_{\varphi,\alpha}(X,f)}{\sharp Aut(f^\ast E\Gamma)} = Z_{\varphi,\alpha}(X) \in \mathbb{F}$.
\end{proof}


\section{Examples}
\label{compute_examples}

\subsection{Untwisted theory}
\label{section_untwisted_theory}

In this subsection, we give some computations for the simplest case : $\hat{\alpha} \cong 0$.
We first give results under the simplest case : $\hat{\alpha} = 0$ strictly.
If we have a strict equality $\hat{\alpha} = 0$, then it is easy to compute $\hat{Z}_{\hat{\varphi},\hat{\alpha}}$.
Since $\hat{Z}_{\hat{\varphi},\hat{\alpha}}$ depends on $\hat{\alpha}$ naturally, we only need to consider the case of $\hat{\alpha} \cong 0$.

Let $\hat{\alpha} \in \hat{F}^{n+2}(B_0 (B\Gamma))$ be the unit. i.e. it is a collapsing morphism between some spectra.
Since $\alpha = 0 \in F^n (B\Gamma)$, it is obvious that $S_{\alpha} (X^n , f) = 0 \in \widetilde{G}_0 (S^0)$.
For a connected closed $E$-oriented $n$-manifold $X^n$, we have
\begin{align}
Z_{\varphi, \alpha} (X) = \frac{\sharp Hom (\pi_1 (X) , \Gamma ) }{\sharp \Gamma}  \in \mathbb{F}.
\end{align}
In particular, for a $E$-oriented closed connected $(n-1)$-dimensional manifold $Y$ such that $E_n (Y) \cong 0$, we have
\begin{align}
dim_{\mathbb{F}} (\hat{Z}_{\hat{\varphi} , \hat{\alpha}} (Y)) = \frac{\sharp Hom (\pi_1 (Y \times S^1) , \Gamma ) }{\sharp \Gamma} .
\end{align}

\subsection{The case of $\Gamma = \mathbb{Z}/2$, $E=F=H\mathbb{Z}/2$}
In that case, we have $F^n (B\Gamma) = H^n (B\mathbb{Z}/2 ; \mathbb{Z}/2 ) = H^n (\mathbb{R}P^{\infty} ; \mathbb{Z}/2 ) \cong \mathbb{Z}/2$ so that we have only one nontrivial class $\alpha = w^n_1$ where $w_1$ denotes the first Stiefel Whitney class.
We write $Z_{\varphi , \alpha}= Z_{w^n_1}$ for short.

For simplicity, let us denote by $\tau_{X}$ the map $H^1 (X; \mathbb{Z}/2) \to H^n (X ; \mathbb{Z}/2) ~;~z \mapsto z^n$ for an $n$-manifold $X$.

Let $n$ be arbitrary natural number.
For closed $n$-manifold $X$ whose first betti number is 1, we have
\begin{align}
Z_{w^n_1} ( X^n) = 
\begin{cases}
1 & \text{$\tau_X = 0$} \\
0 & \text{$\tau_X \neq 0$}
\end{cases}
\end{align}
We give two examples as such manifolds $X$, the real projective space $\mathbb{R}P^n$ and the Dold manifold $P(m,l),~m,l\geq 1$.
Based on this computation, we can compute the dimension of vector spaces.
We have $Z_{w^n_1} (\mathbb{R}P^n) = 0$ and $Z_{w^n_1} (P(m,l)) = 1$ where $m+ 2l = n$.
If the first Betti number of a closed $(n-1)$-manifold $Y$ is zero, then we have
\begin{align}
dim_{\mathbb{F}} \hat{Z}_{w^n_1} (Y) = 1
\end{align}
It is easy to check that if $Y$ is 1-connected, then the linear space $Z_{\hat{\varphi},\alpha}$ is one-dimensional in general, but in this case, we have more general results.

Let us suppose that $n= 2^k$ for some natural number $k$.
For such specific $n$, it is possible to compute all of the Dijkgraaf-Witten invariants of manifolds.
\begin{align}
Z_{w^n_1} ( X^n) =
\begin{cases}
1 & \text{$\tau_X = 0$} \\
0 & \text{$\tau_X \neq 0$}
\end{cases}
\end{align}
For example, we have
\begin{align}
Z_{w^n_1} ( \Sigma_g ) = 2^{2g-1},
\end{align}
where $\Sigma_g$ is the closed surface with genus $g$.
Also we have 
\begin{align}
Z_{w^n_1} (K^{n_1} \times L^{n_2}) = 2^{\beta_1 (K^{n_1}) + \beta_1 (L^{n_2}) -1} ,
\end{align}
where $n_1,n_2$ are nonnegative integers such that $n_1 + n_2 = n$.
Thus, we have 
\begin{align}
dim_{\mathbb{F}} \hat{Z}_{w^n_1} (Y) = 2^{\beta_1 (Y)} .
\end{align}
The above computations follow from the proposition below.

\begin{prop}
Let $X$ be a closed $n$-manifold.
If $\beta_1 (X)=1$ or $n= 2^k$ for some $k$, then we have
\begin{align}
Z_{w^n_1} (X) = \frac{1}{2} \prod_k (1 + \varphi ( \langle [X] , v^n_k \rangle ) ) \in \mathbb{F}
\end{align}
where we take a basis $v_k$'s of $H^1 (X ; \mathbb{Z}/2)$.
\end{prop}
\begin{proof}
$2 \dot Z_{w^n_1} (X) = \sharp \Gamma \cdot Z_{\varphi,\alpha}$ is defined as a sum of $\langle f_\ast [X] , w^n_1 \rangle = \langle [X] , (f^\ast w_1)^n \rangle$ where $f $ ranges over $[X , \mathbb{R}P^{\infty}]$.
If we use the isomorphism $[X , \mathbb{R}P^{\infty} ] \to H^1 (X ; \mathbb{Z}/2) ; [f] \mapsto f^\ast (w_1)$, one sees that $2 \cdot Z_{w^n_1} (X)$  coincides with the sum of $\langle [X], v^n \rangle$ where $v$ ranges over $H^1 (X ; \mathbb{Z}/2)$.
Thus if $\beta_1 (X) =1$, then the claim is true.
If  $n = 2^k$ for some $k$, then we have
\begin{align}
\sum_{v} \varphi (\langle [X] , v^n \rangle )
=
\sum_{a_1,\cdots a_N} \prod_{k} \varphi ( \langle [X] , v^n_k \rangle )^{a_k}
=
\prod_k (1 + \varphi ( \langle [X] , v^n_k \rangle ) ) .
\end{align}
Here, $a_k$ runs over $\mathbb{Z}/2$.
The claim is proved.
\end{proof}

\section{$\hat{\varphi}$ from $\varphi$}
\label{commets_about_rho}

We give some remarks with respect to the symmetric monoidal functor $\hat{\varphi}$.
We put $\hat{\varphi}$.
Note that the symmetric monoidal functor $\hat{\varphi}$ always induces a group homomorphism $\pi_1 (\hat{\varphi}) : G \to M$ by applying the 1st homotopy group of $\hat{\varphi}$.
Here, $G$ is an abelian group $\pi_1 (\hat{G}_{-1}(S^0)) \cong \widetilde{G}_0 (S^0)$.
In this section, we give a sufficient condition that a group homomorphism $\varphi  : G \to M$ naturally induces such a symmetric monoidal functor $\hat{\varphi}$.

\begin{prop}
\label{forrhoconst}
Let $\mathcal{H}$ be a symmetric categorical group.
Set $H$ as the automorphism group $\mathcal{H}(1_\mathcal{H},1_\mathcal{H})$ where $1_{\mathcal{H}}$ is the unit of the symmetric categorical group $\mathcal{H}$.
If $\mathcal{H}$ is 0-connected, then for an object $1_{\mathcal{H}} \in \mathcal{H}$ the follwing assignments form a symmetric monoidal functor :
\begin{align}
\mathbf{T}_{\mathcal{H}} : \mathcal{H} &\to H\mathbf{Tor} \notag \\
a &\mapsto \mathcal{H}(1_{\mathcal{H}},a) \notag \\
(a \stackrel{f}{\to} b) &\mapsto (\mathcal{H}(1_{\mathcal{H}},a) \stackrel{f_\ast}{\to}  \mathcal{H}(1_{\mathcal{H}},b) ) . \notag 
\end{align}
\end{prop}
\begin{proof}
The assignment form a functor $\mathcal{H} \to \mathbf{Set}$ obviously.
Let us determine a $H$-torsor structure of $\mathbf{T}_{\mathcal{H}} (a) = \mathcal{H}(1_{\mathcal{H}} ,a)$ for $a \in \mathcal{H}$.
Due to the monoidal structure of $\mathcal{H}$, we have a map $(g,x) \mapsto g\cdot x$ defined by the following compostions.
\begin{align}
H \times\mathcal{H}(1_{\mathcal{H}} ,a) = \mathcal{H}(1_{\mathcal{H}} , 1_{\mathcal{H}}) \times \mathcal{H}(1_{\mathcal{H}} ,a) \to \mathcal{H}(1_{\mathcal{H}} \otimes 1_{\mathcal{H}} , 1_{\mathcal{H}} \otimes a ) \cong \mathcal{H}(1_{\mathcal{H}} ,a)\notag
\end{align}
Then the coherence of structure morphisms of $\mathcal{H}$ shows that the correspondence $(g,x) \mapsto g\cdot x$ gives an action of $H$ on the set $\mathbf{T}_{\mathcal{H}} (a)$.
We show that this action is free.
If $g \cdot x = x$, then we have the following commutative diagram in $\mathcal{H}$ :
$$
\begin{tikzcd}
1_{\mathcal{H}} \ar[r, "x"] & a & 1_{\mathcal{H}} \ar[l, "x"'] \\
1_{\mathcal{H}} \otimes 1_{\mathcal{H}} \ar[r, "id \otimes x"] \ar[u, "\cong"] & 1_{\mathcal{H}} \otimes a \ar[u, "\cong"] & 1_{\mathcal{H}} \otimes 1_{\mathcal{H}} \ar[l, "g\otimes x"'] \ar[u, "\cong"]
\end{tikzcd}
$$
Hence we obtain $g = id_{1_{\mathcal{H}}} \in \mathcal{H}(1_{\mathcal{H}} , 1_{\mathcal{H}} )$ due to the coherence of structure morphisms.

On the one hand, since $\mathcal{H}(1_{\mathcal{H}} , a)$ is not empty, there exists an isomorphism $\mathcal{H}(1_{\mathcal{H}} , a)\cong \mathcal{H}(1_{\mathcal{H}},1_{\mathcal{H}})$.
Then since the action of $H$ on $\mathcal{H}(1_{\mathcal{H}}, a) \cong \mathcal{H}(1_{\mathcal{H}} ,1_{\mathcal{H}})$ is free, the $H$-set $\mathcal{H}(1_{\mathcal{H}},a)$ is a $H$-torsor.

Finally, we have natural isomorphisms of $H$-torsors by the structure morphisms :
\begin{align}
\mathcal{H}(a_0, a_1) \otimes_{H} \mathcal{H}(b_0 , b_1) \cong \mathcal{H}(a_0 \otimes b_0 , a_1 \otimes b_1) , \notag
\end{align}
where $\otimes_{H}$ denotes the tensor product of $H$-torsors.
Due to this natural isomorphism, $\mathbf{T}_{\mathcal{H}}$ becomes a symmetric monoidal functor.
\end{proof}

Consider the symmetric categorial group $\mathcal{H} = \hat{G}_{-1}(S^0)$.
As a corollary, we have a sufficient condition to obtain a symmetric monoidal functor $\hat{\varphi} : \mathcal{H} \to M\mathbf{Tor}$ starting form a group homomorphism $\varphi : H \to M$ where $H$ is the group $\pi_1 (\mathcal{H} ) \cong \tilde{G}_0 (S^0) $.

\begin{Corollary}
Let $\varphi : G_0(\mathrm{pt})  \to M$ be a group homomorphism.
Suppose that the underlying groupoid of $\hat{G}_{-1}(S^0)$ is 0-connected, i.e. $\pi_0 (\hat{G}_{-1}(S^0)) \cong \pi_{-1}(G) \cong 0$.
Denote by $\times_{\varphi} M : G\mathbf{Tor} \to M\mathbf{Tor}$ the symmetric monoidal functor given by the asscoiated bundle construction.
Then $\varphi$ induces a symmetric monoidal functor $\hat{\varphi} : \hat{G}_{-1}(S^0) \to M\mathbf{Tor}$ via compostions of $\times_{\varphi}M$ and $\mathbf{T}_{\hat{G}_{-1}(S^0)}$ in Propostion \ref{forrhoconst}.
\end{Corollary}


\section{$KK$ categorical group}
\label{section_categorical_KK_theory}

In this section, we construct a categorical group version of $KK$-theory as the fundamental groupoid of a space of quasi-homomorphisms between C*-algebras where its 0-th homotopy set gives $KK$-theory upto isomorphism.
The abelian group structure on $KK$-theory is lifted to that groupoid as a symmetric categorical group structure.

In subsection \ref{KKtheory_scg_prelimi}, we introduce two simplicial sets formed by homomorphisms of C*-algebras and by quasi-homomorphisms of C*-algebras.
In subsection \ref{Categorical_KK}, we define a categorical group version of $KK$-theory as the fundamental groupoid of the simplicial set of quasi-homomorphisms.

\subsection{Preliminaries}
\label{KKtheory_scg_prelimi}

Let $A,B$ be C*-algebras.
A map $f : A \to B$ is a homomorphism if it is a $*$-homomorphism of $*$-algebras $A,B$.
We denote by $hom (A,B)$ the set of homomorphisms from $A$ to $B$.
In this section, we introduce a simplicial set structure `on' the set of homomorphisms $hom(A,B)$.

\begin{Defn}
\label{thick_C}
We define a simplicial set $HOM (A,B)$ for C*-algebras $A,B$.
For $n \in \mathbb{Z}^{\geq 0}$, we define a set of $n$-simplices of $HOM(A,B)$ as follows :
\begin{align}
HOM (A,B)_n \stackrel{\mathrm{def.}}{=} hom (A, C(\triangle (n) , B)) . \notag
\end{align}
Then the face maps and the degeneracy maps between simplices $\triangle (n)$'s induce a simplicial C*-algebra $C(\triangle (n) , B)$'s and a simplicial set $hom (A, C(\triangle^\bullet , B)) = HOM(A,B)$.

\end{Defn}

Then the composition $\circ : hom (A,B) \times hom(B,C) \to hom(A,C)$ and the tensor product $\otimes : hom(A_0, A_1 ) \times hom (B_0 , B_1 ) \to hom (A_0 \otimes B_0 , A_1 \otimes B_1)$ are lifted as simplicial maps between $HOM(-,-)$.
In other words, we have simplicial maps defined by the pointwise composition and the pointwise tensor product.
\begin{align}
&\circ : HOM(A,B) \times HOM(B,C) \to HOM(A,C) , \\
&\otimes : HOM(A_0 , A_1) \times HOM(B_0 ,B_1) \to HOM(A_0 \otimes B_0 , A_1 \otimes B_1) .
\end{align}
Here, the left hand side denotes the direct product of simplicial sets.

We introduce a notation for iterated Cuntz algebra $q^M A$ for a C*-algebra $A$ following M. Joachim and S. Stolz \cite{JS}.
\begin{Defn}
\label{cuntz_iter}
For a C*-algebra $A$ and a finite set $M$, we denote by $Q^M A$ a free product of copies of the C*-algebra $A$ indexed by the power set of $M$ : 
\begin{align}
 \Asterisk_{K \subset M} A . \notag
\end{align}
The iterated Cuntz algebra $q^M A$ is defined as the ideal in $Q^M A$ generated by the elements
\begin{align}
q^M (a) \stackrel{\mathrm{def.}}{=} \sum_{K \subset M} (-1)^{\sharp K} a^K  \in Q^M A ~,~~~a \in A . \notag
\end{align}
\end{Defn}

Let us fix a countably infinite set $\mathbf{U}$.
We construct a direct system $(q^M A , \pi_{M,N};M \subset \mathbf{U} , \sharp M < + \infty)$ for a C*-algebra $A$.
Let $N \subset M$ be finite subsets in the set $\mathbf{U}$.
Let $\pi^\prime_{M,N} : Q^M A \to Q^N A$ be the homomorphism induced by $a^K \mapsto a^K$ if $K\subset N$ and $a^K \mapsto 0$ otherwise.
It induces a homomorphism $\pi_{M,N} : q^M A \to q^N A$.

On the one hand, using a rank-one projection $p \in \mathcal{K}$ we define a direct system $(\mathcal{K}^M \otimes B , p^{M\backslash N} \otimes (-) ;M \subset \mathbf{U},\sharp M < +\infty )$.
Here $\mathcal{K}$ is the C*-algebra formed by compact oeprators on a fixed countable Hilbert space which is infinite dimensional.
It determines a homomorphism $p^{M\backslash N} \otimes (-) : \mathcal{K}^N \otimes B \to \mathcal{K}^M \otimes B$ where $p^{M\backslash N}$ is the projection obtained by taking tensor products of copies of the projection $p$ indexed by the set $M\backslash N$.

The above direct systems induce a direct system $(HOM(q^M A, \mathcal{K}^M \otimes B) , j_{M,N}  ; M \subset \mathbf{U} , \sharp M < + \infty )$ where $j_{M,N} : HOM(q^N A, \mathcal{K}^N \otimes B) \to HOM(q^M A, \mathcal{K}^M \otimes B)$ is defined as follows :
$$
\begin{tikzcd}
HOM(q^N A, \mathcal{K}^N \otimes B) \ar[d, "(p^{M\backslash N}\otimes (-))_\ast"'] & \\
HOM(q^N A, \mathcal{K}^M \otimes B) \ar[r, "\pi_{MN}^\ast" ] &
HOM(q^M A, \mathcal{K}^M \otimes B)
\end{tikzcd}
$$

\begin{Defn}
\label{QAB}
For C*-algebras $A,B$, we define a based simplicial set $Q (A,B)$ satisfying the Kan condition as :
\begin{align}
Q (A,B) \stackrel{\mathrm{def.}}{=}\varinjlim_{M} HOM ( q^M A, \mathcal{K}^M \otimes B  ) . \notag
\end{align}
It is well-defined since every direct system in the category of simplicial sets satisfying the Kan condition has a direct limit.
\end{Defn}

We define tensor product $\otimes : Q(A_0, A_1) \times Q(B_0 , B_1) \to Q(A_0 \otimes B_0 , A_1 \otimes B_1)$ as a simplicial map.
Let $M,N$ be finite subsets of the set $\mathbf{U}$.
It is obtained from simplicial maps which is compatible with each direct system,
\begin{align} 
\otimes : HOM(q^M A_0 , \mathcal{K}^M \otimes A_1 ) \times & HOM(q^N B_0 , \mathcal{K}^N \otimes B_1) \notag \\
&\to HOM(q^{M\amalg N}(A_0 \otimes B_0) , \mathcal{K}^{M \amalg N}(A_1 \otimes B_1) ) \notag
\end{align}
as compositions of the following simplicial maps :
\begin{align}
 &~~~~HOM(q^M A_0 , \mathcal{K}^M \otimes A_1 ) \times HOM(q^N B_0 , \mathcal{K}^N \otimes B_1) \\
 &\stackrel{\otimes}{\to}
 HOM(q^M A_0 \otimes q^N B_0 ,( \mathcal{K}^M \otimes A_1 ) \otimes (\mathcal{K}^N \otimes B_1)) \\
 &\stackrel{\cong}{\to}
 HOM(q^M A_0 \otimes q^N B_0 , \mathcal{K}^{M\amalg N} \otimes B) \\
 &\to
 HOM(q^{M\amalg N}(A_0 \otimes B_0) , \mathcal{K}^{M\amalg N} \otimes B)
\end{align}
Here the final simplicial map is induced by the canonical homomorphism $q^{M\amalg N}(A_0 \otimes B_0) \to q^M A_0 \otimes q^N B_0$.

In a similar way, we construct a composition $\otimes_D : Q(A,D) \times Q(D,B) \to Q(A,B)$ as a simplicial map.
For finite subsets $M,N$ in the set $\mathbf{U}$, we define a simplicial map 
\begin{align}
 \otimes_D : HOM (q^M A, \mathcal{K}^M \otimes D) \times HOM ( q^N D, \mathcal{K}^N \otimes B) \to HOM (q^{M\amalg N}A, \mathcal{K}^{M\amalg N}\otimes B) \notag
\end{align}
as compositions of the following simplicial maps :
\begin{align}
&~~~~~~HOM (q^M A, \mathcal{K}^M \otimes D) \times HOM ( q^N D, \mathcal{K}^N \otimes B) \\
&\stackrel{q^N \times (\mathcal{K}^M \otimes (-))}{\to} 
HOM (q^N(q^M A), q^N(\mathcal{K}^M \otimes D)) \times HOM ( \mathcal{K}^M \otimes q^N D, \mathcal{K}^M\otimes (\mathcal{K}^N \otimes B)) \\
&\stackrel{\chi^{MN}}{\to}
HOM (q^N(q^M A), \mathcal{K}^M \otimes q^N D) \times HOM ( \mathcal{K}^M \otimes q^N D, \mathcal{K}^M\otimes (\mathcal{K}^N \otimes B)) \\
&\stackrel{\circ}{\to}
HOM (q^N(q^M A), \mathcal{K}^M \otimes (\mathcal{K}^N \otimes B)) \\
&\stackrel{\cong}{\to}
 HOM (q^{M\amalg N}A, \mathcal{K}^{M\amalg N}\otimes B)
\end{align}
The second simplicial map $\chi^{MN}$ is induced by the canonical homomorphism $q^N(\mathcal{K}^M \otimes D) \to \mathcal{K}^M \otimes q^N D$.

\subsection{Construction}
\label{Categorical_KK}

In this section, we introduce a categorical group version of $KK$-theory.
Its underlying groupoid $\mathcal{KK}(A,B)$ is defined as the fundamental groupoid of the simplicial set $Q(A,B)$ for C*-algebras $A,B$.
Each element of $z \in \mathbf{U}$ induces a categorical group structure on the groupoid where $\mathbf{U}$ is used to define $Q(A,B)$.
The composition and the tensor product are defined.
We explain how $\mathcal{KK}(A,B)$ and $KK(A,B)$ are related to each other.

In subsection \ref{categorical cohomology group}, we introduce a categorical group version of generalized (co)homology theory.
Its monoidal structure essentially comes from the fact that the pinch map $S^n \to S^n \vee S^n$ gives a comultiplication on $S^n$.
In analogy to it, we shall introduce a monoidal structure on $\mathcal{KK}(A,B)$ using the fact that a homomorphism $\mathcal{K} \oplus \mathcal{K} \to \mathcal{K}$ gives a multiplication on $\mathcal{K}$ in some sense (see Proposition \ref{K_sym_monoid}).
Here, $\mathcal{K}$ is a C*-algebra formed by compact operators on a fixed infinite dimensional separable Hilbert space.

The assignment $(A,B ) \mapsto \Pi_1 HOM(A,B) $ for arbitrary C*-algebras $A,B$ induces a 2-category $\Pi_1 \mathbf{C}^\ast_{\bullet}$ :
The 2-category $\Pi_1\mathbf{C}^\ast_{\bullet}$ consists of the following data subject to the axioms of 2-categories :
\begin{enumerate}
\item
The class of objects consists of C*-algebras.
\item
For two objects $A,B$, a groupoid $\Pi_1 HOM(A,B)$ is given, which is the collection of morphisms.
$\Pi_1 HOM(A,B)$ denotes the fundamental groupoid of the simplicial set $HOM(A,B)$ satisfying the Kan condition.
\item
For three objects $A,B,D$, there is a functor which gives the composition :
\begin{align}
\circ= \circ_{A,D,B} : \Pi_1 HOM(A,D) \times \Pi_1 HOM(D,B) \to \Pi_1 HOM(A,B) .\notag
\end{align}
\item
For an object $A$, there is a simplicial map $1_A : \ast \to \Pi_1 HOM(A,A)$ which gives the identity with respect to the above compostion.
Here, $\ast$ is a fixed one-point groupoid.
\end{enumerate}

\begin{prop}
\label{K_sym_monoid}
Let $\mathcal{K}$ be the C*-algebra of compact operators on an infinite dimensional separable Hilbert space $\mathbf{H}$.
Let us consider $\mathcal{K}$ as an object of $\Pi_1 \mathbf{C}^\ast_{\bullet}$.
A unitary isomorphism $\mathbf{H} \oplus \mathbf{H} \cong \mathbf{H}$ of separable Hilbert spaces gives $\mathcal{K}$ a structure of symmetric monoid in $(\Pi_1 \mathbf{C}^\ast_{\bullet} , \oplus )$.
Here $(\Pi_1 \mathbf{C}^\ast_{\bullet} , \oplus )$ is considered as the symmetric monoid in the 2-category $\mathbf{2Cat}$.
\end{prop}
\begin{proof}
Let $w : \mathbf{H} \oplus \mathbf{H} \cong \mathbf{H}$ be a unitary isomorphism.
It induces two unitary isomorphisms $w_{((12)3)} , w_{(1(23))} : \mathbf{H} \oplus \mathbf{H} \oplus \mathbf{H} \to \mathbf{H}$ according to the order of applying $w : \mathbf{H} \oplus \mathbf{H} \to \mathbf{H}$, for example, $w_{((12)3)} $ is defined as follows :
\begin{align}
\mathbf{H} \oplus \mathbf{H} \oplus \mathbf{H} \stackrel{w \oplus id_{\mathbf{H}}}{\to} \mathbf{H} \oplus \mathbf{H} \stackrel{w}{\to} \mathbf{H} \notag
\end{align}
Then there is a continuous path $\gamma (t)$ in the space of unitary isomorphisms from $\mathbf{H} \oplus \mathbf{H} \oplus \mathbf{H} $ to $\mathbf{H}$ under the norm topology such that $\gamma (0) = w_{((12)3)}$ and $\gamma (1) = w_{(1(23))}$.

Note that the unitary isomorphism $w$ induces a homomorphism :
\begin{align}
\mu : \mathcal{K}\oplus \mathcal{K} \stackrel{j}{\to} M_2 (\mathcal{K}) \cong \mathcal{K}(\mathbf{H} \oplus \mathbf{H}) \stackrel{Ad(w)}{\to} \mathcal{K} \notag
\end{align}
Here $j (x,y) = \begin{pmatrix} x & 0 \\ 0 & y \end{pmatrix}$.
We consider the homomorphism $\mu$ as a morphism in the category $\Pi_1 \mathbf{C}^\ast_{\bullet}$.
Then the continuous path $\gamma (t)$ chosen above gives an associator for $\mu$.
In fact, $\gamma (t)$ gives a path in the simplicial set  $HOM(\mathcal{K}\oplus \mathcal{K} \oplus \mathcal{K}, \mathcal{K})$ which connects $\mu_{((12)3}$ and $\mu_{(1(23)}$.
It determines a morphism in $\Pi_1 HOM(\mathcal{K}\oplus \mathcal{K} \oplus \mathcal{K}, \mathcal{K})$ which connects its objects $\mu_{((12)3}$ and $\mu_{(1(23)}$.
We denote it as $\alpha$.
Note that $\alpha$ does not depend on the choice of $\gamma$ since for two different $\gamma (t)$, $\gamma^\prime (t)$ they are homotopic to each other preserving their boundary values due to the Kuiper's theorem.
Then $\alpha : \mu_{((12)3)} \to \mu_{(1(23))}$ satisfies the pentagon identity due to Kuiper's theorem again.

We consider the zero homomorphism $0 \to \mathcal{K}$ as a morphism $u: 0\to \mathcal{K}$ in the 2-category $\Pi_1 \mathbf{C}^\ast_{\bullet}$.
From now on we show that $u$ gives a unit for $\mu$.
We take a continuous path $\gamma^\prime (t)$ in the space of isometries on $\mathbf{H}$ under the strong operator topology, from $id_{\mathbf{H}} : \mathbf{H} \to \mathbf{H}$ to the isometry $\mathbf{H} \stackrel{id_{\mathbf{H}} \oplus 0}{\to}  \mathbf{H} \oplus \mathbf{H} \stackrel{w}{\to} \mathbf{H}$.
We claim that $\gamma^\prime (t)$ induces a right unit cancellaration for $\mu$.
In fact, via its adjoint homomorphism, $\gamma^\prime (t)$ induces a path in the simplicial set $HOM(\mathcal{K}, \mathcal{K})$ which connects $id_{\mathcal{K}}$ to $\mu \circ (id_{\mathcal{K}} \oplus u)$.
It determines a morphism in $\Pi_1 HOM(\mathcal{K}, \mathcal{K})$ which connects its objects $id_{\mathcal{K}}$ to $\mu \circ (id_{\mathcal{K}} \oplus u)$.
We denote the morphism as $r$.
Similarly, we define a left cancellaration $l$ which is a morphism in $\Pi_1 HOM(\mathcal{K}, \mathcal{K})$ which connects its objects $id_{\mathcal{K}}$ to $\mu \circ (u\oplus id_{\mathcal{K}})$.
\end{proof}

\begin{Corollary}
A unitary isomorphism $\mathbf{H} \oplus \mathbf{H} \cong \mathbf{H}$ gives the groupoid $\Pi_1 HOM(A,B \otimes \mathcal{K})$ symmetric monoidal groupoid structure naturally.
\end{Corollary}
\begin{proof}
From the object $\mu \in \Pi_1 HOM( \mathcal{K} \oplus \mathcal{K} ,\mathcal{K})$, we obtain a functor :
\begin{align}
\Pi_1 HOM(A,B \otimes \mathcal{K}) \times \Pi_1 HOM(A,B \otimes \mathcal{K}) &\stackrel{\oplus}{\to} \Pi_1 HOM(A,B \otimes (\mathcal{K} \oplus \mathcal{K})) \notag \\
&\stackrel{\mu\circ}{\to} \Pi_1 HOM(A,B \otimes \mathcal{K}) \notag
\end{align}
Then $\alpha : \mu_{((12)3)} \to \mu_{(1(23))}$, $r: \mu\circ (id_{\mathcal{K}} \oplus u) \to id_{\mathcal{K}}$, $l:\mu\circ (u \oplus id_{\mathcal{K}} ) \to id_{\mathcal{K}} $ in the proof of Propostion \ref{K_sym_monoid} induce a structure of symmetric monoidal groupoid on the groupoid $\Pi_1 HOM(A, B\otimes \mathcal{K})$.
\end{proof}

By far, we showed that the fundamental groupoid $\Pi_1 HOM(A, B\otimes \mathcal{K})$ has a symmetric monoidal groupoid structure.
From now on, we will show that if the C*-algebra $A$ is a Cuntz algebra associated with a C*-algebra, then $\Pi_1 HOM(A, B\otimes \mathcal{K})$ becomes a symmetric categorical group.

\begin{Defn}
If $(M,m)$ is a based set, then the C*-algebra $Q^M A$ in Definition \ref{cuntz_iter} has a natural involution $t$ defined as follows.
Then for $K\subset M$, let us define homomorphisms $f_K : A \to QA$ by $a^K \mapsto a^{K\backslash \{m\} }$ if $m \in K$ and $a^K \mapsto a^{K \cup \{ m \}}$ if $m \notin K$.
Then the family $f_K~;~K\subset M$ gives a homomorphism $Q^M A \to Q^M A$.
We define its restriction to $q^M A$ and denote it by $t$, which is obviously an invoultion.
\end{Defn}

\begin{Lemma}
\label{t_homotopic_tozero}
For a homomorphism $f \in hom (qA, B)$, the homomorphism obtained from the following compositions is homtopic to the zero homomorphism.
In particular, we can take a canonical one as such homotopy.
\begin{align}
qA \stackrel{\Delta}{\to} qA \oplus qA \stackrel{f \oplus (f\circ t)}{\to} B \oplus B \stackrel{j}{\to} M_2 (B) \notag
\end{align}

\end{Lemma}
\begin{proof}
The claim is equivalent with that $qA \to M_2 (qA) ; x \mapsto \begin{pmatrix} x & 0 \\  0 & t(x) \end{pmatrix}$ is homotopic to the zero homomorphism.
For $t \in I$, homomorphisms $g : A \to M_2 (QA); a \mapsto \begin{pmatrix} a^0 & 0 \\ 0 & a^1 \end{pmatrix}$ and $A \to M_2 (QA) ; a \mapsto R(t) \begin{pmatrix} a^0 & 0 \\ 0 & a^1 \end{pmatrix} R(t)^\ast$ induce a homomorphism $h_t : QA \to M_2 (QA)$ by the universality of $QA$.
Here we set $R(t) = \begin{pmatrix} cos ( \frac{\pi}{2}t) & sin( \frac{\pi}{2}t) \\ -sin( \frac{\pi}{2}t) & cos ( \frac{\pi}{2}t) \end{pmatrix}$.
Moreover, we have $h_t (x) \in M_2 (qA)$ for $x \in qA \subset QA$.
By the continuity of $h_t$ with respect to $t$, $h_t : qA \to M_2 (qA)$ induces $h_0 \simeq h_1 $ where $h_1 $ is the above homomorphism $g$ and $h_0 = 0$.

\end{proof}

\begin{Lemma}
\label{t_homotopic_tozero1}
The involution $t : qA \to qA$ and the homotopy from $j \circ (id_{qA} \oplus t)\circ \Delta_{qA} \simeq 0 : qA \to M_2 (qA)$ in the proof of Lemma \ref{t_homotopic_tozero} give a canonical inverse  over the symmetric monoidal groupoid $\Pi_1 HOM (qA, \mathcal{K} \otimes B)$.
\end{Lemma}
\begin{proof}
The homomorphism $t : qA \to qA$ induces a based map $hom (qA, C(I^n , \mathcal{K} \otimes B)) \to hom(qA, C(I^n, \mathcal{K} \otimes B))$ which is an involution for $n \geq 0$.
Then it induces a based simplicial map $\tau : HOM(qA, \mathcal{K} \otimes B) \to HOM(qA, \mathcal{K} \otimes B)$ which is also an involution.

Let us denote by $h$ the homotopy from $j \circ (id_{qA} \oplus t)\circ \Delta_{qA} \simeq 0 : qA \to M_2 (qA)$ in the proof of Lemma \ref{t_homotopic_tozero}.
Then $h$ gives a morphism $l_A $ in $\Pi_1 HOM(qA, M_2 (qA))$.
Consider the following composition :
\begin{align}
\Pi_1 HOM(qA, M_2 (qA)) \times \Pi_1 HOM(M_2 (qA) , M_2 (\mathcal{K} \otimes B)) \to \Pi_1 HOM(qA, M_2 (\mathcal{K} \otimes B)) \notag
\end{align}
We obtain a natural isomorphism $c$ from $(j \circ (id_{qA} \oplus t)\circ \Delta_{qA} )^\ast $ to the constant functor valued at the zero homomorphism.
Since the composition $\mu \circ (1_{HOM(qA,\mathcal{K} \otimes B)} \times \tau)\circ \Delta_{HOM(qA,\mathcal{K} \otimes B)}$ coincides with the following compositions,
\begin{align}
&HOM(qA, \mathcal{K} \otimes B) \stackrel{id_{M_2(\mathbb{C})}\otimes}{\to} HOM(M_2(\mathbb{C})\otimes qA  , M_2 (\mathbb{C}) \otimes \mathcal{K} \otimes B) \notag \\
&\cong HOM(M_2 (qA), M_2 (\mathcal{K} \otimes B)) 
\stackrel{(j \circ (id_{qA} \oplus t)\circ \Delta_{qA} )^\ast}{\longrightarrow} HOM (qA, M_2(\mathcal{K} \otimes B)) \notag \\
&\stackrel{(Ad(w))_\ast}{\to} HOM(qA, \mathcal{K} \otimes B) \notag
\end{align}
the natural isomorphism $c$ induces a natural isomorphism from $\mu \circ (1_{HOM(qA,\mathcal{K} \otimes B)} \times \tau)\circ \Delta_{HOM(qA,\mathcal{K} \otimes B)}$ to the constant functor valued at the zero homomorphism.
It completes the proof.
\end{proof}

\begin{Corollary}
\label{based_H_space_stru}
A based finite set $(M,m)$ induces a structure of symmetric categorical group on the fundamental groupoid $\Pi_1 HOM (q^M A , \mathcal{K}^M \otimes B)$ which is natural with respect to $A,B$.
\end{Corollary}
\begin{proof}
By applying Lemma \ref{t_homotopic_tozero1} to $HOM(q^M A, \mathcal{K}^M \otimes B) \cong HOM(q (q^{M\backslash\{m \}}A), \mathcal{K} \otimes (\mathcal{K}^{M\backslash\{m \}} \otimes B))$, the claim is proved.
\end{proof}

\begin{Defn}
\label{SCG_ver_KK_defn}
We define a groupoid $\mathcal{KK}(A,B)$ as follows :
\begin{align}
\mathcal{KK} (A,B) \stackrel{\mathrm{def.}}{=}  \Pi_1 Q ( A,B) . \notag
\end{align}

\end{Defn}

\begin{prop}
\label{z_KK_SCGstr}
An element $z \in \mathbf{U}$ naturally gives a symmetric categorical group structure the groupoid $\mathcal{KK} (A,B)$.
Also it has canonical inverses.
\end{prop}
\begin{proof}
We have a simplicial isomorphism :
\begin{align}
\varinjlim_{(M,z)} HOM ( q^M A, \mathcal{K}^M \otimes B  ) \to Q(A,B)  . \notag
\end{align}
Here, $(M,z)$ in the left hand side runs on the directed set consisting of finite subsets of $\mathbf{U}$ based at $z\in \mathbf{U}$.
We obtain a SCG structure on the fundamental groupoid $\Pi_1 Q(A,B)$ of the simplicial set $Q(A,B)$.
\end{proof}

\begin{prop}
\label{SCG_kasparov_produ}
For a bijection of sets $\mathbf{U} \amalg \mathbf{U} \cong \mathbf{U}$, we have a functor natural with respect to C*-algebras $A,B,D$ :
\begin{align}
\otimes_D : \mathcal{KK} (A,D) \times \mathcal{KK} (D,B) \to \mathcal{KK} (A,B)  . \notag
\end{align}

\end{prop}
\begin{proof}
Recall that we constructed a natural pairing $Q(A,D) \times Q(D,B) \to Q(A,B)$ in subsection \ref{KKtheory_scg_prelimi} where $A,B,C$ are C*-algebras.
By definitions it induces a functor $\otimes_D : \mathcal{KK}(A,D) \times \mathcal{KK}(D,B) \to \mathcal{KK}(A,B)$ which is natural with respect to $A,B,D$.
\end{proof}

By definitions, we have a natural isomorphism for C*-algebras $A,B$.
\begin{align}
\pi_n (HOM (A,B)) \cong [A, S^n  B] . \notag
\end{align}
Here $[A,B]$ denotes the set of homotopy classes of homomorphisms from $A$ to $B$.
Thus, considering the Cuntz picture of $KK$-theory, we obtain a natural isomorphisms of groups for C*-algebras $A,B$ :
\begin{align}
\pi_0 (\mathcal{KK} (A,B)) &\cong KK (A,B) , \notag \\
\pi_1 (\mathcal{KK} (A,B)) &\cong KK(A,S B) . \notag
\end{align}
Here, we consider the unit object of $\mathcal{KK}(A,B)$ as the basepoint to take $\pi_1 (\mathcal{KK}(A,B))$, but it holds for arbitrary basepoints since $\mathcal{KK}(A,B)$ has a symmetric categorical group structure.
Under these isomorphisms, the pairing $\otimes_D$ in Propostion \ref{SCG_kasparov_produ} induces the Kasparov product $\otimes_D : KK (A,D) \times KK (D,B) \to KK (A,B)$.

In the following statements, we explain `bilinearity' of the functor $\mathcal{KK}(-,-)$ with respect to the direct sum of C*-algebras.
Proposition \ref{B_oplus_oplus}, \ref{A_oplus_oplus} correspond to Proposition \ref{E_vee_oplus}, \ref{F_vee_oplus} in generalized (co)homology settings respectively.
\begin{prop}
\label{B_oplus_oplus}
Let $A,B,B^\prime$ be C*-algebra.
Let $B\oplus B^\prime \to B$ and $B\oplus B^\prime \to B^\prime$ be the canonical projections.
They induce an isomorphism of groupoids :
\begin{align}
\label{KKK_directsum_isom}
\mathcal{KK} (A, B\oplus B^\prime) \to \mathcal{KK} (A, B) \times \mathcal{KK} (A, B^\prime) .
\end{align}

\end{prop}
\begin{proof}
The maps induced by applying $\pi_0, \pi_1$ to (\ref{spectra_wedge_isom}) are isomorphisms.
Hence (\ref{KKK_directsum_isom}) gives an equivalence of groupoids.
Moreover, it is obvious that the functor induces a bijection between object classes by definition of $\mathcal{KK}(-,-)$.
Therefore, (\ref{KKK_directsum_isom}) is an isomorphism of groupoids.
\end{proof}

Let $A,A^\prime,B$ be C*-algebras.
Denote by the canonical inclusions $i : A \to A \oplus A^\prime$, $i^\prime : A^\prime \to A \oplus A^\prime$.
They induce a functor $f$ by
\begin{align}
\mathcal{KK} (A ,B) \oplus \mathcal{KK} (A^\prime , B) &\to \mathcal{KK} (A \oplus A^\prime ,B) \notag \\
(a,b) &\mapsto i^\ast (a) \oplus (i^\prime)^\ast (b) \notag 
\end{align}
If $p : A \oplus A^\prime \to A$ and $p^\prime : A \oplus A^\prime \to A^\prime$ denotes the projections, we obtain a functor $g$, 
\begin{align}
\mathcal{KK}  (A \oplus A^\prime , B) &\to \mathcal{KK} (A,B) \oplus \mathcal{KK} (A^\prime , B) \notag \\
a &\mapsto (p^\ast (a) , (p^\prime)^\ast (a)) \notag
\end{align}

\begin{prop}
\label{A_oplus_oplus}
The functors $f,g$ are lifted to an adjoint equivalence of symmetric categorical groups with canonical inverses.
\end{prop}
\begin{proof}
The proof proceeds in a parallel way to the proof of Proposition \ref{F_vee_oplus}.
\end{proof}

\begin{remark}
There is a Puppe exact sequence of $KK$-theory (see section 19 \cite{Blackadar}).
It is an exact sequence associated with two C*-algebras, a homomorphism of them and its cone.
We also have a Puppe exact sequence on the symmetric categorical group version of $KK$-theory in a similar way with Remark \ref{rel2exactsequ_remark}.
\end{remark}


\bibliography{A_generalization_of_DW_theory}{}
\bibliographystyle{alpha}

\end{document}